\documentclass[10pt]{article}
\usepackage{etex}
\usepackage[a4paper]{geometry}
\usepackage{amsmath,amsfonts,amssymb,amsthm,mathrsfs,calc}
\usepackage[english]{babel}

\usepackage[usenames,dvipsnames]{xcolor}
\usepackage{tikz}
\usepackage[all]{xy}
\usepackage{caption}
\usepackage{graphicx}
\usepackage{mathrsfs}
\usepackage{pstricks}
\usepackage{pstricks-add}
\usepackage{appendix}

\psset{unit=1mm}

\usepackage[colorlinks=false,linkbordercolor=white,citebordercolor=white]{hyperref}

\usepackage{enumitem}
\setlist{nolistsep}

  \theoremstyle{plain}

\newtheorem{proposition}{Proposition}[section]
\newtheorem{corollary}[proposition]{Corollary}

  \theoremstyle{remark}
\newtheorem{remark}[proposition]{Remark}
  \theoremstyle{definition}
\newtheorem{definition}[proposition]{Definition}
\newtheorem{notation}[proposition]{Notation}
\newtheorem{convention}[proposition]{Convention}
\newtheorem{question}[proposition]{Question}
\newtheorem{lemma}[proposition]{Lemma}
\newtheorem{example}[proposition]{Example}

\newcommand*{\wlhd}{\mathrel{\ensuremath{\widetilde{\lhd}}}}
\newcommand{\op}{\diamond}

\newcommand{\NN}{\mathbb{N}}
\newcommand{\ZZ}{\mathbb{Z}}

\newcommand{\RR}{\mathbb{R}}

\newcommand{\C}{\mathcal C}
\newcommand{\PP}{\mathcal P}

\renewcommand{\S}{\mathcal S}
\newcommand{\BW}{\mathcal{W}}
\newcommand{\ww}{\omega}
\newcommand{\w}{\chi}
\newcommand{\wl}{\lambda}
\newcommand{\wlinv}{\rotatebox[origin=c]{180}{\reflectbox{\ensuremath{\lambda}}}}
\newcommand{\f}{\varphi}
\renewcommand{\b}{\varepsilon}
\newcommand{\ob}{\widetilde{\b}}
\newcommand{\gen}{\theta}
\newcommand{\rrho}{\tau}
\newcommand{\sq}{\varsigma}
\newcommand{\D}{D}
\newcommand{\A}{\mathscr A}
\newcommand{\B}{\mathscr B}

\renewcommand{\Col}{\mathscr C}
\newcommand{\ColIs}{\Col^{iso}}
\newcommand{\FP}{Fix}
\newcommand{\St}{Stab}
\newcommand{\SF}{SF}
\newcommand{\ASF}{ABF}
\newcommand{\asf}{\chi}
\newcommand{\X}{X}
\newcommand{\XX}{\mathscr X}
\renewcommand{\L}{\Lambda}
\newcommand{\LL}{\mathscr L}
\renewcommand{\SS}{S}
\renewcommand{\P}{P}
\newcommand{\Q}{Q}
\newcommand{\G}{G}
\newcommand{\Cuff}{C}
\newcommand{\bw}{{\color{white}b}}

\newcommand{\Hom}{\operatorname{Hom}}
\newcommand{\Aut}{\operatorname{Aut}}
\newcommand{\Id}{\operatorname{Id}}

\renewcommand{\le}{\leqslant}
\renewcommand{\ge}{\geqslant}

\newcommand*\rcircled[1]{$\,$\tikz[baseline=(char.base)]{
            \node[shape=circle,draw,inner sep=1pt] (char) {#1};}}

\begin{document}

\title{Qualgebras and knotted $3$-valent graphs}

\author{Victoria Lebed \\ \itshape lebed.victoria@gmail.com}

\maketitle

\begin{abstract}
\footnotesize This paper is devoted to qualgebras and squandles, which are quandles enriched with a compatible binary/unary operation. Algebraically, they are modeled after groups with conjugation and multiplication/squaring operations. Topologically, qualgebras emerge as an algebraic counterpart of knotted $3$-valent graphs, just like quandles can be seen as an ``algebraization'' of knots; squandles in turn simplify the qualgebra algebraization of graphs. Knotted $3$-valent graph invariants are constructed by counting qualgebra/squandle colorings of graph diagrams, and are further enhanced using qualgebra/squandle $2$-cocycles. Some algebraic properties and the beginning of a cohomology theory are given for both structures. A classification of size $4$ qualgebras/squandles is presented, and their second cohomology groups are completely described.
\end{abstract}



\section{Introduction}\label{sec:intro}

A \textit{quandle} is a set~$\Q$ endowed with two binary operations~$\lhd$ and~$\wlhd$ satisfying the following axioms:
\begin{align}
&\text{R}\mathrm{III} &\text{\textit{self-distributivity}: }\qquad\qquad & (a\lhd b)\lhd c = (a\lhd c)\lhd(b\lhd c), \label{E:SD}\tag{$Q_{SD}$}\\
&\text{R}\mathrm{II}  &\text{\textit{invertibility}: }\qquad\qquad & (a\lhd b){\wlhd} b = (a{\wlhd} b)\lhd b = a,\label{E:Inv}\tag{$Q_{Inv}$}\\
&\text{R}\mathrm{I} &\text{\textit{idempotence}: }\qquad\qquad & a\lhd a = a. \label{E:Idem}\tag{$Q_{Idem}$}
\end{align}
Since operation~$\wlhd$ can be deduced from~$\lhd$ using~\eqref{E:Inv}, we shall often omit it from the definition. Originating from the work of topologists D.Joyce and S.Matveev \cite{Joyce,Matveev}, this structure can be seen as an \textbf{algebraic counterpart of knots}. Indeed, consider colorings of the arcs of knot diagrams by elements of~$\Q$, according to the rule on Figure~\ref{pic:Colorings}\rcircled{A}. This coloring rule is compatible with Reidemeister moves (Figure~\ref{pic:RMoves}) if and only if Axioms \eqref{E:SD}-\eqref{E:Idem} are verified, each axiom corresponding to the Reidemeister move indicated in the left column above. Thus the number of diagram colorings by a fixed quandle defines an invariant of underlying knots and links. This invariant can be strengthened by endowing each colored crossing --- and hence, summing everything together, each diagram coloring --- with a weight (Figure~\ref{pic:Quandles}). The weights are calculated using a \textit{quandle $2$-cocycle} of~$\Q$ according to a procedure suggested by Carter-Jelsovsky-Kamada-Langford-Saito (\cite{QuandleHom}).
\begin{center}
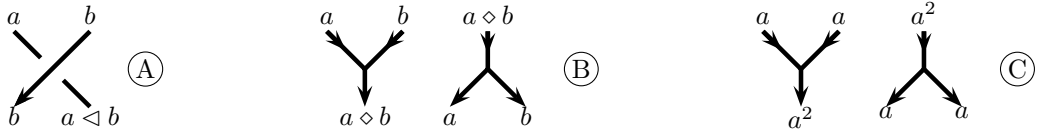

\begin{pspicture}(40,15)
\psline[linewidth=0.6](0,10)(10,0)
\psline[linewidth=0.6,border=1.8,arrowsize=2]{->}(10,10)(0,0)
\rput[b](0,11){$a$}
\rput[b](10,11){$b$}
\rput[t](0,0){$b$}
\rput[t](10,0){$a \lhd b$}
\rput(17,5){\rcircled{A}}
\end{pspicture}
\begin{pspicture}(15,15)
\psline[linewidth=0.6,arrowsize=2]{->}(0,10)(2.5,7.5)
\psline[linewidth=0.6,arrowsize=2]{->}(0,10)(5,5)(5,0)
\psline[linewidth=0.6,arrowsize=2]{->}(10,10)(7.5,7.5)
\psline[linewidth=0.6](10,10)(5,5)
\rput[b](0,11){$a$}
\rput[b](10,11){$b$}
\rput[t](5,0){$a \op b$}
\end{pspicture}
\begin{pspicture}(40,15)
\psline[linewidth=0.6,arrowsize=2]{->}(5,10)(5,7)
\psline[linewidth=0.6,arrowsize=2]{->}(5,10)(5,5)(0,0)
\psline[linewidth=0.6,arrowsize=2]{->}(5,5)(10,0)
\rput[b](5,11){$a \op b$}
\rput[t](0,0){$\bw a \bw$}
\rput[t](10,0){$b$}
\rput(17,5){\rcircled{B}}
\end{pspicture}
\begin{pspicture}(15,15)
\psline[linewidth=0.6,arrowsize=2]{->}(0,10)(2.5,7.5)
\psline[linewidth=0.6,arrowsize=2]{->}(0,10)(5,5)(5,0)
\psline[linewidth=0.6,arrowsize=2]{->}(10,10)(7.5,7.5)
\psline[linewidth=0.6](10,10)(5,5)
\rput[b](0,11){$a$}
\rput[b](10,11){$a$}
\rput[t](5,0){$a^2$}
\end{pspicture}
\begin{pspicture}(20,15)
\psline[linewidth=0.6,arrowsize=2]{->}(5,10)(5,7)
\psline[linewidth=0.6,arrowsize=2]{->}(5,10)(5,5)(0,0)
\psline[linewidth=0.6,arrowsize=2]{->}(5,5)(10,0)
\rput[b](5,11){$a^2$}
\rput[t](0,0){$a$}
\rput[t](10,0){$a$}
\rput(17,5){\rcircled{C}}
\end{pspicture}
\medskip
\captionof{figure}{Colorings by quandles, qualgebras and squandles}
\label{pic:Colorings}
\end{center}
From the algebraic viewpoint, the quandle structure can be regarded as an \textbf{axiomatization of the conjugation operation} in a group. Concretely, a group with the conjugation operation $a \lhd b = b^{-1}ab$ is always a quandle, and all the properties of conjugation that hold in every group are consequences of \eqref{E:SD}-\eqref{E:Idem}.

\newpage %
The purpose of this paper is to find an \textbf{algebraic counterpart of knotted $3$-valent graphs} (further simply called \textit{graphs} for brevity) which would develop the quandle ideas. To this end,  we introduce the \textit{qualgebra} structure. It is a quandle $(\Q, \lhd)$ endowed with an additional binary operation~$\op$ satisfying
\begin{align}
&\text{R}\mathrm{IV} &\text{\textit{translation composability}: }\qquad\qquad &a\lhd (b\op c) = (a\lhd b)\lhd c,\label{E:QA1}\tag{$QA_{Comp}$}\\
&\text{R}\mathrm{VI}&\text{\textit{distributivity}: }\qquad\qquad &(a\op b)\lhd c = (a\lhd c) \op (b\lhd c),\label{E:QA2}\tag{$QA_D$}\\
&\text{R}\mathrm{V}&\text{\textit{semi-commutativity}: }\qquad\qquad &a\op b = b \op (a\lhd b).\label{E:QAComm}\tag{$QA_{Comm}$}
\end{align}
Restricting oneself to \textit{well-oriented} graphs (i.e., having only zip and unzip vertices, cf. Figure~\ref{pic:Zip}) and extending the quandle coloring rules~\ref{pic:Colorings}\rcircled{A} to $3$-valent vertices as shown on Figure~\ref{pic:Colorings}\rcircled{B}, one gets rules compatible with Reidemeister moves for graphs (Figure~\ref{pic:RMovesGraphs}) if and only if Axioms \eqref{E:QA1}-\eqref{E:QAComm} are satisfied, each axiom corresponding to the Reidemeister move indicated on the left. Imitating what was done for quandle colorings of knots, one can thus define qualgebra \textit{counting invariants} for graphs. The latter can be upgraded to \textit{weight invariants} using the qualgebra $2$-cocycles introduced in this work. \textit{Qualgebra $2$-cocycles} consist of two maps, one of which is used for putting weights on crossings, and the other one for putting weights on $3$-valent vertices (Figures~\ref{pic:Quandles} and~\ref{pic:BoltzmannGraph}); the weight of a colored diagram is obtained, as usual, by summing everything together.

A group with the conjugation quandle operation becomes a qualgebra with the group multiplication as additional operation: $a \op b = ab$. Algebraically, the additional qualgebra axioms encode \textbf{the relations between conjugation and multiplication operations} in a group (see Table~\ref{tab:3levels}). Note that, however, our qualgebra axioms do not imply any of those used in the standard definition of a group. In particular, we shall give examples of $4$-element qualgebras for which the operation~$\op$ is non-cancellative, non-associative, and has no neutral element.

Besides defining qualgebras and constructing counting and weight invariants of graphs out of them, in this work we study some basic properties of qualgebras; give a complete classification of $4$-element qualgebras (showing that a single quandle can be the base of numerous qualgebra structures with significantly different properties); and suggest the beginning of a qualgebra cohomology theory, calculating in particular the second cohomology group for $4$-element qualgebras. Moreover, we compute certain qualgebra counting and weight invariants for some pairs of graphs, showing that these graphs can be distinguished using our methods.

In parallel with the qualgebra structure, we study the closely related \textit{squandle} structure. It is defined as a quandle $(\Q, \lhd)$ endowed with an additional unary operation~$a \mapsto a^2$, obeying the following axioms (modeled after the properties of conjugation and squaring operations in a group): 
\begin{align}
&\text{R}\mathrm{IV} & \qquad a\lhd b^2 &= (a\lhd b)\lhd b, \qquad\qquad\qquad \label{E:SQA1}\tag{$SQ_1$}\\
&\text{R}\mathrm{VI} & a^2\lhd b &= (a\lhd b)^2.\label{E:SQA2}\tag{$SQ_2$}
\end{align}
A qualgebra with the squaring operation $a^2 = a \op a$ is an example of squandle. The coloring rule from Figure~\ref{pic:Colorings}\rcircled{C} allows to construct invariants of graphs by counting squandle colorings of their diagrams; weight invariants are obtained with the help of \textit{squandle $2$-cocycles}.

The terms ``qualgebra'' and ``squandle'' both come from the names of the two operations participating in the definition of these structures, zipped together as indicated on Figure~\ref{pic:Term}.
\begin{center}
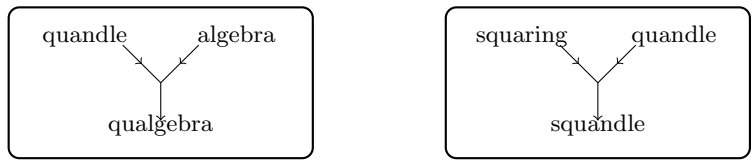

\begin{tikzpicture}
\draw [rounded corners,thick]  (-1.5,-0.5) -- (-1.5,-1.5) -- (2.5,-1.5) -- (2.5,0.5) -- (-1.5,0.5) -- (-1.5,-0.5);
\draw [->] (0,0)--(0.25,-0.25);
\draw [->] (0.25,-0.25)--(0.5,-0.5)--(0.5,-1);
\draw [->] (1,0)--(0.75,-0.25);
\draw (0.75,-0.25)--(0.5,-0.5);
\node at (0.5,-0.8) [below]  {\small {qualgebra}};
\node at (-0.5,-0.15) [above]  {\small {quandle}};
\node at (1.5,-0.15) [above]  {\small {algebra}};
\node at (4,-0.15) {};
\end{tikzpicture}
\begin{tikzpicture}
\draw [rounded corners,thick]  (-1.5,-0.5) -- (-1.5,-1.5) -- (2.5,-1.5) -- (2.5,0.5) -- (-1.5,0.5) -- (-1.5,-0.5);
\draw [->] (0,0)--(0.25,-0.25);
\draw [->] (0.25,-0.25)--(0.5,-0.5)--(0.5,-1);
\draw [->] (1,0)--(0.75,-0.25);
\draw (0.75,-0.25)--(0.5,-0.5);
\node at (0.5,-0.8) [below]  {\small {squandle}};
\node at (-0.5,-0.15) [above]  {\small {squaring}};
\node at (1.5,-0.15) [above]  {\small {quandle}};
\end{tikzpicture}
\captionof{figure}{The terms ``qualgebra'' and ``squandle''} \label{pic:Term}
\end{center}

The paper is organized as follows. The language of colorings, used throughout this paper, is developed in Section~\ref{sec:Colorings}. It is illustrated with the famous example of quandle colorings of knot diagrams, from which some of our further constructions take inspiration. We then turn to invariants of graphs which extend the quandle invariants of knots. In Section~\ref{sec:GraphInvar}, after a brief survey of such extensions found in the literature, we propose an original one based on qualgebra colorings. Our invariants are defined for well-oriented graphs only, but they are shown to induce invariants of unoriented graphs. We further show that groups give an important source of qualgebra examples. Constructions from \cite{IshiiMCQ} and \cite{DehornoyLDM,DrapalLDM,DehornoyFreeALDS}, close to but different from ours, are also discussed. The notion of squandle is introduced in  Section~\ref{sec:Isosceles}, motivated by the concept of special colorings (with isosceles qualgebra colorings as the major example here). Squandle colorings are then used for distinguishing Kinoshita-Terasaka and standard $\Theta$-curves. Section~\ref{sec:Qualgebras4} contains a short study of basic properties of qualgebras and squandles, applied to a complete classification of qualgebras/squandles with $4$ elements. One of the ``exotic'' structures obtained is next used for distinguishing two cuff graphs. Section~\ref{sec:QualgebraHom} is devoted to the notions of qualgebra/squandle $2$-cocycles and $2$-coboundaries, as well as to the induced weight invariants of graphs. Qualgebra/squandle $2$-cocycles and second cohomology groups are calculated for $4$-element structures. The last section contains several suggestions for a further development of the qualgebra ideas presented here. 

\subsubsection*{Acknowledgements}

The author is grateful to Seiichi Kamada and J\'{o}zef Przytycki for stimulating discussions, and to Arnaud Mortier for his comments on an earlier version of this manuscript. During the writing of this paper, the author was supported by a JSPS Postdoctral Fellowship For Foreign Researchers and by JSPS KAKENHI Grant 25$\cdot$03315.

\section{Colorings: generalities and the quandle example}\label{sec:Colorings}

One of the most natural and efficient methods of constructing
invariants of certain topological objects (such as knots, braids,
tangles, knotted graphs, knotted surfaces, etc.) consists in studying colorings of
their diagrams by certain sets of colors. If the coloring rules are
carefully chosen, one can extract invariants of underlying
topological objects by studying diagram colorings --- for instance, considering
their total number, or some more sophisticated coloring
characteristics. In this section we develop a general framework for
such coloring invariants and illustrate it with the celebrated
example of quandle colorings for knots. We prefer a narrative style to a list of definitions here for the sake of readability. The rest of the paper is devoted to several applications of these coloring ideas to knotted $3$-valent graphs.

\subsection*{Topological colorings, counting invariants and quandles}

Let us now fix a class of $1$-dimensional \textit{diagrams} on a surface (e.g., familiar knot diagrams in~$\RR^2$). For this class of diagrams, choose several types of \textit{special points}, with the local picture of a diagram around a special point being determined by the point type (crossing points, points of local maximum and graph vertices are typical examples). These local pictures are called \textit{type patterns} (see Figure~\ref{pic:Colorings} for the examples of oriented crossing point and $3$-valent vertex patterns).
We want to study diagrams up to special-point-preserving isotopy, and up to a set of local (i.e., realized inside a small ball) invertible moves, called \textit{R-moves} (the example inspiring the name is that of Reidemeister moves for knots, cf. Figure~\ref{pic:RMoves}). Diagrams related by isotopy and R-moves are called \textit{R-equivalent}. This defines an equivalence relation on the set of diagrams, which corresponds in the cases of interest to the isotopy equivalence for underlying topological objects.

\begin{center}
\begin{pspicture}(-7.5,-2.5)(17,13)
\pscircle[linewidth=0.2,linestyle=dotted](0,5){7.5}
\pscurve[linewidth=0.2](0,12.5)(0,5)(2,2)(4,4)(4,5)
\pscurve[linewidth=0.2,border=0.5](4,5)(4,6)(2,8)(0,5)(0,-2.5)
\rput(13,5){$\overset{\text{R}\mathrm{I}}{\longleftrightarrow}$}
\end{pspicture}
\begin{pspicture}(-7.5,-2.5)(15,13)
\pscircle[linewidth=0.2,linestyle=dotted](0,5){7.5}
\psline[linewidth=0.2](0,12.5)(0,-2.5)
\end{pspicture}
\begin{pspicture}(-7.5,-2.5)(17,13)
\pscircle[linewidth=0.2,linestyle=dotted](0,5){7.5}
\pscurve[linewidth=0.2](2,-2)(-2,5)(2,12)
\pscurve[linewidth=0.2,border=0.5](-2,-2)(2,5)(-2,12)
\rput(13,5){$\overset{\text{R}\mathrm{II}}{\longleftrightarrow}$}
\end{pspicture}
\begin{pspicture}(-7.5,-2.5)(15,13)
\pscircle[linewidth=0.2,linestyle=dotted](0,5){7.5}
\pscurve[linewidth=0.2](2,-2)(1,5)(2,12)
\pscurve[linewidth=0.2,border=0.5](-2,-2)(-1,5)(-2,12)
\end{pspicture}
\begin{pspicture}(-7.5,-2.5)(17,13)
\pscircle[linewidth=0.2,linestyle=dotted](0,5){7.5}
\pscurve[linewidth=0.2,border=0.5](3,-1.5)(1.5,5)(-3,11.5)
\pscurve[linewidth=0.2,border=0.5](0,-2.5)(-1.5,5)(0,12.5)
\pscurve[linewidth=0.2,border=0.5](3,11.5)(1.5,5)(-3,-1.5)
\rput(13,5){$\overset{\text{R}\mathrm{III}}{\longleftrightarrow}$}
\end{pspicture}
\begin{pspicture}(-7.5,-2.5)(15,13)
\pscircle[linewidth=0.2,linestyle=dotted](0,5){7.5}
\pscurve[linewidth=0.2,border=0.5](3,-1.5)(-1.5,5)(-3,11.5)
\pscurve[linewidth=0.2,border=0.5](0,-2.5)(1.5,5)(0,12.5)
\pscurve[linewidth=0.2,border=0.5](3,11.5)(-1.5,5)(-3,-1.5)
\end{pspicture}
\medskip

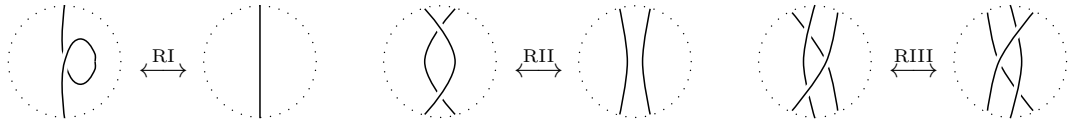
\captionof{figure}{Reidemeister moves for knot diagrams}
\label{pic:RMoves}
\end{center}

 An \textit{arc} is a part of a diagram delimited by special points. Fix a set~$\SS$ (possibly with some algebraic structure), which we think of as the coloring set. An \textit{$\SS$-coloring} of a diagram~$\D$ is a map
$$ \C: \A(\D) \longrightarrow \SS$$
from the set of its arcs to~$\SS$, satisfying some prescribed \textit{coloring rules} for arcs around special points. The set of such colorings of~$\D$ is denoted by~$\Col_{\SS}(\D)$. The notion of $\SS$-coloring extends from our class of diagrams to that of sub-diagrams (for instance, those involved in an R-move) in the obvious way. In the pictures, an arc~$\alpha$ is often decorated with its \textit{color} $\C(\alpha)$. 

\begin{definition}\label{D:TopColRules}
$\SS$-coloring rules are called \emph{topological} if for any (sub-)diagram~$\D$, any $\C \in \Col_{\SS}(\D)$ and any~$\D'$ obtained from~$\D$ by applying one R-move, there exists a unique coloring $\C' \in \Col_{\SS}(\D')$ coinciding with~$\C$ outside the small ball where the R-move was effectuated.
\end{definition} 

Such coloring rules allow one to construct invariants under R-equivalence. The most basic ones are \textit{counting invariants}:

\begin{lemma}\label{L:CountInvar}
Fix a class of diagrams, a set~$\SS$ and topological $\SS$-coloring rules. For any R-equivalent diagrams~$\D$ and~$\D'$, there exists a (non-canonical) bijection between their $\SS$-coloring sets:
\begin{equation}\label{E:ColBij}
\Col_{\SS}(\D) \overset{bij}{\longleftrightarrow} \Col_{\SS}(\D').
\end{equation}
In particular, the function $\D \mapsto \#\Col_{\SS}(\D)$ (where one allows the value~$\infty$) is well-defined on R-equivalence classes of diagrams.
\end{lemma}

Thus, if R-equivalence of diagrams corresponds to the isotopy equivalence for underlying topological objects, the lemma produces invariants of these topological objects.

\begin{proof}
If~$\D$ and~$\D'$ differ by a single R-move, one can take the bijection from the definition of topological coloring rules.
Composing these bijections, one gets the result for the case when~$\D$ and~$\D'$ differ by several R-moves.
\end{proof}

Before giving an example of topological coloring rules, we need a convention concerning orientations:

\begin{convention}\label{C:orientation}
In a class of oriented diagrams, using unoriented strands in R-moves or coloring rules means imposing these moves or rules for all possible orientations.
\end{convention}

\begin{example}\label{EX:QuandleCol}
Consider the class of oriented knot diagrams in~$\RR^2$, crossing points as the only type of
special points, Reidemeister moves from Figure~\ref{pic:RMoves} as R-moves, a set~$\Q$ endowed with a binary operation~$\lhd$ as the coloring set, and $\Q$-coloring rules from Figure~\ref{pic:Colorings}\rcircled{A}. From the
pioneer papers \cite{Joyce,Matveev}, these rules are known to be topological if and only if the structure $(\Q,\lhd)$ is a \textit{quandle}, i.e., satisfies Axioms \eqref{E:SD}-\eqref{E:Idem} (each of which corresponds to one Reidemeister move).
A typical example consists of a group $\G$ with the conjugation
operation $a\lhd b = b^{-1}ab$, called \textit{conjugation
quandle}. Counting invariants for such colorings even by simplest
finite quandles~$\Q$ appear to be rich and efficiently computable.
Note also that they are easily generalized to the diagrams of links
and tangles, as well as to their virtual versions.
\end{example}

\subsection*{Weight invariants and quandle $2$-cocycles}

Let us return to the general setting of a class of diagrams endowed with topological $\SS$-coloring rules. Counting invariants, though already very powerful for quandle colorings of knots, do not exploit the full potential of the bijection from~\eqref{E:ColBij}. More information can be extracted out of it using the following concept:

\begin{definition}\label{D:BW}
A \emph{weight function}~$\ww$ is a collection of maps, one for each
type of special points on our class of diagrams, associating an
integer to any $\SS$-colored pattern of the corresponding type. The
\emph{$\ww$-weight} of an $\SS$-colored (sub-)diagram $(\D,\C)$, denoted by $\BW_{\ww}(\D,\C)$, is the sum of the values of~$\ww$ on all its special points (we suppose the number of
the latter finite). If for any R-move the $\ww$-weights of the two involved sub-diagrams  correspondingly $\SS$-colored (in the sense of Definition~\ref{D:TopColRules}) coincide,
then $\ww$ is called a \emph{Boltzmann weight function}.
\end{definition}

Boltzmann weight functions allow to upgrade counting invariants to
what we call here \textit{weight invariants}:

\begin{lemma}\label{L:WeightedInvar}
Fix a class of diagrams, a set~$\SS$, topological $\SS$-coloring
rules and a Boltzmann weight function~$\ww$. Then the multi-sets of
$\ww$-weights of any R-equivalent diagrams~$\D$ and~$\D'$ coincide:
\begin{equation}\label{E:BWBij}
\{\BW_{\ww}(\D,\C) | \C \in \Col_{\SS}(\D)\} = \{\BW_{\ww}(\D',\C') | \C' \in \Col_{\SS}(\D')\}.
\end{equation}
In particular, restricted to the diagrams~$\D$ for which the set $\Col_{\SS}(\D)$ is finite, the function $$\D \mapsto \sum_{\C \in \Col_{\SS}(\D)} t^{\BW_{\ww}(\D,\C)}\in \ZZ[t^{\pm 1}]$$
is well-defined on R-equivalence classes of diagrams.
\end{lemma}

\begin{proof}
If~$\D$ and~$\D'$ differ by a single R-move, then
Definition~\ref{D:TopColRules} describes a bijection between
$\Col_{\SS}(\D)$ and  $\Col_{\SS}(\D')$ such that corresponding
colorings~$\C$ and~$\C'$ differ only in small balls where the R-move
is effectuated; Definition~\ref{D:BW} then gives $\BW_{\ww}(\D,\C) =
\BW_{\ww}(\D',\C')$, implying the desired multi-set equality. Iterating this argument, one gets the result for the case when~$\D$ and~$\D'$ differ by several R-moves.
\end{proof}

Note that Equality~\eqref{E:BWBij}, as well as most further examples and results, remain valid
if weight functions are allowed to take values in any Abelian group and not only in the group of integers~$\ZZ$.

\begin{example}\label{EX:QuandleCol2}
Continuing Example~\ref{EX:QuandleCol}, take a map~$\w:\Q\times\Q\rightarrow\ZZ$ and consider a weight function, still denoted by~$\w$, that depends only on two of the colors around a crossing point (which is the only type of special points here) as shown on Figure~\ref{pic:Quandles}. In~\cite{QuandleHom} this weight function was shown to be Boltzmann if and
only if it satisfies the following axioms for all elements of~$\Q$
(corresponding, respectively, to moves R$\mathrm{III}$
and R$\mathrm{I}$, the remaining one being automatic):
\begin{align}
&\w(a, b) +\w(a\lhd b, c) = \w(a\lhd c,b\lhd c) + \w(a, c), \label{E:BWR3}\\
&\w(a, a)= 0. \label{E:BWR1}
\end{align}
Moreover, these conditions were interpreted as the definition of
\textit{$2$-cocycles} from the celebrated \textit{quandle cohomology}
theory. In this theory, \textit{$2$-coboundaries} are defined by
\begin{align}
&\w_\f(a, b) = \f(a\lhd b) - \f(a) \label{E:BWB}
\end{align}
for any map $\f:\Q \rightarrow \ZZ$, and they are precisely the $2$-cocycles such that $\BW_\w$ vanishes on all $\Q$-colored knot diagrams.

\begin{center}
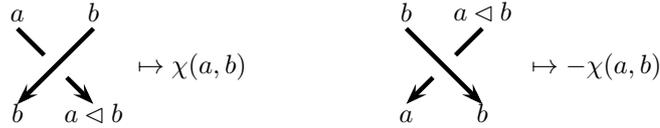

\begin{pspicture}(50,13)
\psline[linewidth=0.6,arrowsize=2]{->}(0,10)(10,0)
\psline[linewidth=0.6,border=1.8,arrowsize=2]{->}(10,10)(0,0)
\rput[b](0,11){$a$}
\rput[b](10,11){$b$}
\rput[t](0,0){$b$}
\rput[t](10,0){$a \lhd b$}
\rput(23,5){$\mapsto \w(a,b)$}
\end{pspicture}
\begin{pspicture}(40,13)
\psline[linewidth=0.6,arrowsize=2]{->}(10,10)(0,0)
\psline[linewidth=0.6,border=1.8,arrowsize=2]{->}(0,10)(10,0)
\rput[b](0,11){$b$}
\rput[b](10,11){$a \lhd b$}
\rput[t](0,0){$\bw a \bw$}
\rput[t](10,0){$b$}
\rput(25,5){$\mapsto -\w(a,b)$}
\end{pspicture}
\medskip
\captionof{figure}{Quandle $2$-cocycle weight function for knot diagrams} \label{pic:Quandles}
\end{center}

Weight invariants of knots constructed out of quandle $2$-cocycles
are known as \textit{quandle cocycle invariants}. They are even more
efficient than quandle counting invariants, since the same small
quandle can admit various $2$-cocycles. Moreover, they are strictly
stronger than quandle counting invariants since, contrary to the
latter, they can distinguish a knot from its mirror image. See \cite{QuandleHom,Kamada,CarterComputations,CarterDiagrammatic,NelsonLN,PrzytyckiTrefoilQuandle} and references therein for more details.
\end{example}

\section{Qualgebra coloring invariants of knotted $3$-valent graphs}\label{sec:GraphInvar}

We now turn to our main object of study, namely, to \textit{knotted $3$-valent graphs} (i.e., embeddings of abstract $3$-valent graphs into~$\RR^3$) and their diagrams in~$\RR^2$; see Figures~\ref{pic:theta} and~\ref{pic:cuffs} for typical examples. In what follows, the word ``graph'' is often used instead of ``knotted $3$-valent graph'' for brevity. Two types of special points are relevant for graph diagrams: crossing points and graph vertices. In 1989, L.H.Kauffman, S.Yamada and D.N.Yetter independently \cite{KauffmanGraphs,YamadaGraphs,YetterGraphs} extended the Reidemeister moves for knots (Figure~\ref{pic:RMoves}) by the three moves presented on Figure~\ref{pic:RMovesGraphs}, showing that the resulting 6 moves precisely describe  graph isotopy in $\RR^3$. We therefore choose them as R-moves here, noting that R-equivalence classes of graph diagrams now correspond to isotopy classes of represented graphs. The names of the moves are chosen here to visually resemble the sub-diagrams involved.

\begin{center}
\begin{pspicture}(-7.5,-2.5)(17,13)
\pscircle[linewidth=0.2,linestyle=dotted](0,5){7.5}
\pscurve[linewidth=0.2](3,-1.5)(-1.5,1.5)(-3,11.5)
\pscurve[linewidth=0.2](0.5,5)(-0.3,9)(0,12.5)
\pscurve[linewidth=0.2](0.5,5)(2.5,8)(3,11.5)
\psline[linewidth=0.2,border=0.6](0.5,5)(-3,-1.5)
\rput(13,5){$\overset{\text{R}\mathrm{IV}}{\longleftrightarrow}$}
\end{pspicture}
\begin{pspicture}(-7.5,-2.5)(15,13)
\pscircle[linewidth=0.2,linestyle=dotted](0,5){7.5}
\pscurve[linewidth=0.2](3,-1.5)(2,4)(1,6)(-3,11.5)
\pscurve[linewidth=0.2,border=0.6](-1.5,1.5)(-1,7)(0,12.5)
\pscurve[linewidth=0.2,border=0.6](-1.5,1.5)(2,4)(3,11.5)
\pscurve[linewidth=0.2](-1.5,1.5)(-1,7)(0,12.5)
\psline[linewidth=0.2](-1.5,1.5)(-3,-1.5)
\end{pspicture}
\begin{pspicture}(-7.5,-2.5)(17,13)
\pscircle[linewidth=0.2,linestyle=dotted](0,5){7.5}
\pscurve[linewidth=0.2](-3,11.5)(2,5)(0,2)
\pscurve[linewidth=0.2,border=0.6](3,11.5)(-2,5)(0,2)
\psline[linewidth=0.2](1,3)(0,2)(0,-2.5)
\rput(13,5){$\overset{\text{R}\mathrm{V}}{\longleftrightarrow}$}
\end{pspicture}
\begin{pspicture}(-7.5,-2.5)(15,13)
\pscircle[linewidth=0.2,linestyle=dotted](0,5){7.5}
\psline[linewidth=0.2](0,2)(0,-2.5)
\pscurve[linewidth=0.2](-3,11.5)(-2,5)(0,2)
\pscurve[linewidth=0.2](3,11.5)(2,5)(0,2)
\end{pspicture}
\begin{pspicture}(-7.5,-2.5)(17,13)
\pscircle[linewidth=0.2,linestyle=dotted](0,5){7.5}
\pscurve[linewidth=0.2](-0.5,5)(0.3,9)(0,12.5)
\pscurve[linewidth=0.2](-0.5,5)(-2.5,8)(-3,11.5)
\psline[linewidth=0.2](-0.5,5)(3,-1.5)
\pscurve[linewidth=0.2,border=0.6](-3,-1.5)(1.5,1.5)(3,11.5)
\rput(13,5){$\overset{\text{R}\mathrm{VI}}{\longleftrightarrow}$}
\end{pspicture}
\begin{pspicture}(-7.5,-2.5)(15,13)
\pscircle[linewidth=0.2,linestyle=dotted](0,5){7.5}
\pscurve[linewidth=0.2](1.5,1.5)(1,7)(0,12.5)
\pscurve[linewidth=0.2](1.5,1.5)(-2,4)(-3,11.5)
\psline[linewidth=0.2](1.5,1.5)(3,-1.5)
\pscurve[linewidth=0.2,border=0.6](-3,-1.5)(-2,4)(-1,6)(3,11.5)
\end{pspicture}
\medskip

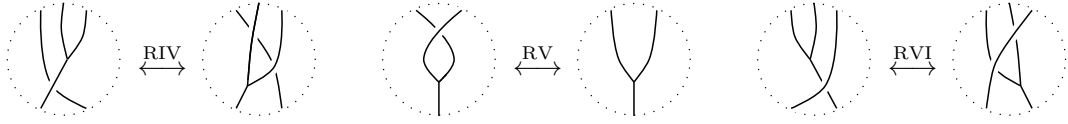
\captionof{figure}{Additional Reidemeister moves for knotted $3$-valent graph diagrams} \label{pic:RMovesGraphs}
\end{center}

Since quandles worked so well for knots, we would like to use a quandle $(\Q, \lhd)$ as the coloring set in the generalized setting of graphs as well. This section is thus devoted to the following question:

\begin{question}
How can one extend the $\Q$-coloring rule from Figure~\ref{pic:Colorings}\rcircled{A} to $3$-valent vertices so that the resulting coloring rules for graphs are topological?
\end{question}

After a short discussion of existing answers, we shall propose an original one. Since the coloring rule around crossing points will always be that from Figure~\ref{pic:Colorings}\rcircled{A} in this paper, we shall often omit it, restricting our study to rules around $3$-valent vertices.

\subsection*{Colorings for graphs: existing approaches}

Required coloring rules are easy to define geometrically for a conjugation quandle
$(\G, a \lhd b = b^{-1}ab)$.
Choose a basepoint~$p$ situated ``over'' a diagram~$\D$ of an oriented graph~$\Gamma$. Consider the  \textit{Wirtinger presentation} of the graph group $\pi_1(\RR^3\backslash \Gamma; p)$  with one generator~$\gen_\alpha$ for each arc~$\alpha$ of~$\D$, constructed according to Figure~\ref{pic:QuandlesForGraphs}\rcircled{A}. An (evident) relation is imposed on the generators around each special point. A representation of  $\pi_1(\RR^3\backslash \Gamma; p)$ in~$\G$ is now a map~$\PP$ from $\{\gen_\alpha | \alpha \in \A(\D)\}$ to~$\G$ respecting these relations. But for~$\PP$ to respect these relations is precisely the same thing as for the map $\C:\alpha \mapsto \PP(\gen_\alpha)$ to be a coloring with respect to coloring rules from Figures~\ref{pic:Colorings}\rcircled{A} and~\ref{pic:QuandlesForGraphs}\rcircled{B} (where in the relation a color
or its inverse should be chosen according to the arc being directed from or to the graph vertex).
The latter coloring rules are topological, as can be seen via this graph group representation interpretation, or by an easy direct verification. For any diagram~$\D$ of~$\Gamma$, one thus gets a bijection 
$$\Col_{\G}(\D) \overset{bij}{\longleftrightarrow} \Hom(\pi_1(\RR^3\backslash \Gamma),\G).$$

These conjugation quandle colorings for graphs can be generalized in several ways.
First, in 2010 M.Niebrzydowski \cite{NiebrzydowskiGraphs} extended the rules from Figure~\ref{pic:QuandlesForGraphs}\rcircled{B} to general
quandles, as shown on Figure~\ref{pic:QuandlesForGraphs}\rcircled{C} (here and
afterwards notation $\lhd^{+}$ stands for~$\lhd$, and $\lhd^{-}$
stands for~$\wlhd$; the choice in~$\pm$ depends, as usual, on
orientations). Another approach was proposed by A.Ishii in his
recent preprint \cite{IshiiMCQ}. He considered a quandle operation~$\lhd$ on a
disjoint union of groups $X = \bigsqcup_i \G_i$, which is the conjugation
operation when restricted to each~$\G_i$ and which satisfies some
additional conditions. Such structure is called a \textit{multiple
conjugation quandle (MCQ)}, and it includes as particular cases usual
conjugation quandles and \textit{$G$-families of quandles}, defined in 2012
by Ishii-Iwakiri-Jang-Oshiro \cite{Gfamilies}. The coloring rule from
Figure~\ref{pic:QuandlesForGraphs}\rcircled{B}, where one demands $a,b$
and~$c$ to lie in the same group~$\G_i$, is topological for MCQ.

\begin{center}
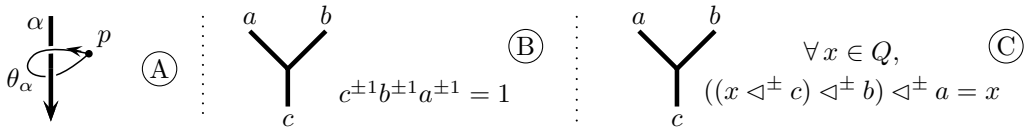

\begin{pspicture}(-7,0)(25,13)
\psline[linewidth=0.6](0,12)(0,0)
\pscurve[linewidth=0.2,border=0.4](5,7)(-2,7)(-3,5)(0,4)(5,7)
\psline[linewidth=0.2,arrowsize=1.5]{->}(4,7.3)(2,7.7)
\psline[linewidth=0.6,border=0.4,arrowsize=2]{->}(0,5)(0,-2)
\pscircle*(5,7){0.5}
\rput(7,9){$p$}
\rput(-2,11){$\alpha$}
\rput[t](-4,5){$\gen_\alpha$}
\rput(14,5){\rcircled{A}}
\psline[linestyle=dotted](20,-1)(20,12)
\end{pspicture}
\begin{pspicture}(50,15)
\psline[linewidth=0.6](0,10)(5,5)(5,0)
\psline[linewidth=0.6](10,10)(5,5)
\rput[b](0,11){$a$}
\rput[b](10,11){$b$}
\rput[t](5,-1){$c$}
\rput(23,2){$c^{\pm 1}b^{\pm 1}a^{\pm 1} = 1$}
\rput(36,8){\rcircled{B}}
\psline[linestyle=dotted](43,-2)(43,12)
\end{pspicture}
\begin{pspicture}(50,15)
\psline[linewidth=0.6](0,10)(5,5)(5,0)
\psline[linewidth=0.6](10,10)(5,5)
\rput[b](0,11){$a$}
\rput[b](10,11){$b$}
\rput[t](5,-1){$c$}
\rput(28,7){$\forall \, x \in \Q,$}
\rput(28,2){$((x \lhd^{\pm} c) \lhd^{\pm} b) \lhd^{\pm} a =x$}
\rput(48,8){\rcircled{C}}
\end{pspicture}
\captionof{figure}{Possible extensions of quandle colorings to graph
diagrams} \label{pic:QuandlesForGraphs}
\end{center}

\subsection*{Well-oriented $3$-valent graphs}

The coloring rule we introduce in this work is another generalization of conjugation quandle colorings of graphs to a broader class of quandles. It is defined for graphs oriented in a special way:

\begin{definition}\label{D:WellOriented}
An abstract or knotted oriented $3$-valent graph is called  \emph{well-oriented} if it has only \textit{zip} and \textit{unzip} vertices, cf. Figure~\ref{pic:Zip}.
\end{definition}

In other words, one forbids source and sink vertices.

\begin{center}
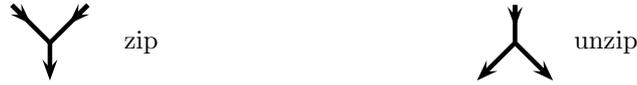

\begin{pspicture}(60,12)
\psline[linewidth=0.6,arrowsize=2]{->}(0,10)(2.5,7.5)
\psline[linewidth=0.6,arrowsize=2]{->}(0,10)(5,5)(5,0)
\psline[linewidth=0.6,arrowsize=2]{->}(10,10)(7.5,7.5)
\psline[linewidth=0.6](10,10)(5,5)
\rput(17,5){zip}
\end{pspicture}
\begin{pspicture}(20,12)
\psline[linewidth=0.6,arrowsize=2]{->}(5,10)(5,7)
\psline[linewidth=0.6,arrowsize=2]{->}(5,10)(5,5)(0,0)
\psline[linewidth=0.6,arrowsize=2]{->}(5,5)(10,0)
\rput(17,5){unzip}
\end{pspicture}
\captionof{figure}{Zip and unzip vertices for $3$-valent graphs} \label{pic:Zip}
\end{center}

For well-oriented graph diagrams, some of the R-moves can be discarded using the so called \textit{Turaev's trick} (see also \cite{PolyakRMoves} for a detailed and careful study of minimal generating sets of Reidemeister moves in the knot case):

\begin{lemma}\label{L:RMovesForWellOrGraphs}
Reidemeister moves $\mathrm{IV}$-$\mathrm{VI}$ with orientations as in Figure~\ref{pic:RMovesForWellOrGraphs}, together with all oriented versions of moves R$\mathrm{I}$-R$\mathrm{III}$, imply all remaining well-oriented versions of moves R$\mathrm{IV}$-R$\mathrm{VI}$.
\end{lemma}

\begin{center}
\begin{pspicture}(-7.5,-2.5)(17,13)
\pscircle[linewidth=0.2,linestyle=dotted](0,5){7.5}
\pscurve[linewidth=0.2,arrowsize=1.5]{->}(-3,11.5)(-1.5,1.5)(3,-1.5)
\psline[linewidth=0.2,arrowsize=1.5]{->}(3,11.5)(2.5,9)
\pscurve[linewidth=0.2](2.5,9)(2.3,8)(1.5,6)(0.5,5)
\psline[linewidth=0.2,arrowsize=1.5]{->}(0,12.5)(-0.5,10)
\pscurve[linewidth=0.2](-0.5,10)(-0.6,9.5)(-0.4,7)(0.5,5)
\psline[linewidth=0.2,border=0.6,arrowsize=1.5]{->}(0.5,5)(-3,-1.5)
\rput(13,5){$\overset{\text{R}\mathrm{IV}^z}{\longleftrightarrow}$}
\end{pspicture}
\begin{pspicture}(-7.5,-2.5)(15,13)
\pscircle[linewidth=0.2,linestyle=dotted](0,5){7.5}
\pscurve[linewidth=0.2,arrowsize=1.5]{->}(-3,11.5)(1,6)(2,4)(3,-1.5)
\psline[linewidth=0.2,arrowsize=1.5]{->}(3,11.5)(2.8,8.5)
\pscurve[linewidth=0.2,border=0.6](2.8,8.5)(2,4)(-1.5,1.5)
\psline[linewidth=0.2,arrowsize=1.5]{->}(0,12.5)(-0.5,10)
\pscurve[linewidth=0.2,border=0.6](-0.5,10)(-1.2,6)(-1.5,1.5)
\pscurve[linewidth=0.2](2.8,8.5)(2,4)(-1.5,1.5)
\psline[linewidth=0.2,arrowsize=1.5]{->}(-1.5,1.5)(-3,-1.5)
\end{pspicture}
\begin{pspicture}(-7.5,-2.5)(17,13)
\pscircle[linewidth=0.2,linestyle=dotted](0,5){7.5}
\pscurve[linewidth=0.2,arrowsize=1.5]{->}(3,-1.5)(-1.5,1.5)(-3,11.5)
\pscurve[linewidth=0.2,arrowsize=1.5]{->}(0.5,5)(-0.3,9)(0,12.5)
\pscurve[linewidth=0.2,arrowsize=1.5]{->}(0.5,5)(2.5,8)(3,11.5)
\psline[linewidth=0.2,border=0.6,arrowsize=1.5]{->}(-3,-1.5)(0.5,5)
\rput(13,5){$\overset{\text{R}\mathrm{IV}^u}{\longleftrightarrow}$}
\end{pspicture}
\begin{pspicture}(-7.5,-2.5)(15,13)
\pscircle[linewidth=0.2,linestyle=dotted](0,5){7.5}
\pscurve[linewidth=0.2,arrowsize=1.5]{->}(3,-1.5)(2,4)(1,6)(-3,11.5)
\pscurve[linewidth=0.2,border=0.6,arrowsize=1.5]{->}(-1.5,1.5)(-1,7)(0,12.5)
\pscurve[linewidth=0.2,border=0.6,arrowsize=1.5]{->}(-1.5,1.5)(2,4)(3,11.5)
\pscurve[linewidth=0.2](-1.5,1.5)(-1,7)(0,12.5)
\psline[linewidth=0.2,arrowsize=1.5]{->}(-3,-1.5)(-1.5,1.5)
\end{pspicture}
\begin{pspicture}(-7.5,-2.5)(17,13)
\pscircle[linewidth=0.2,linestyle=dotted](0,5){7.5}
\pscurve[linewidth=0.2,arrowsize=1.5]{->}(-3,11.5)(2,5)(0,2)
\pscurve[linewidth=0.2,border=0.6,arrowsize=1.5]{->}(3,11.5)(-2,5)(0,2)
\psline[linewidth=0.2,arrowsize=1.5]{->}(1,3)(0,2)(0,-2.5)
\rput(13,5){$\overset{\text{R}\mathrm{V}^z}{\longleftrightarrow}$}
\end{pspicture}
\begin{pspicture}(-7.5,-2.5)(15,13)
\pscircle[linewidth=0.2,linestyle=dotted](0,5){7.5}
\psline[linewidth=0.2,arrowsize=1.5]{->}(0,2)(0,-2.5)
\pscurve[linewidth=0.2,arrowsize=1.5]{->}(-3,11.5)(-2,5)(0,2)
\pscurve[linewidth=0.2,arrowsize=1.5]{->}(3,11.5)(2,5)(0,2)
\end{pspicture}
\medskip

\begin{pspicture}(-7.5,-2.5)(17,13)
\pscircle[linewidth=0.2,linestyle=dotted](0,5){7.5}
\pscurve[linewidth=0.2,arrowsize=1.5]{->}(0,2)(2,5)(-3,11.5)
\pscurve[linewidth=0.2,border=0.6,arrowsize=1.5]{->}(0,2)(-2,5)(3,11.5)
\psline[linewidth=0.2,arrowsize=1.5]{->}(0,-2.5)(0,2)
\psline[linewidth=0.2](0,2)(1,3)
\rput(13,5){$\overset{\text{R}\mathrm{V}^u}{\longleftrightarrow}$}
\end{pspicture}
\begin{pspicture}(-7.5,-2.5)(15,13)
\pscircle[linewidth=0.2,linestyle=dotted](0,5){7.5}
\psline[linewidth=0.2,arrowsize=1.5]{->}(0,-2.5)(0,2)
\pscurve[linewidth=0.2,arrowsize=1.5]{->}(0,2)(-2,5)(-3,11.5)
\pscurve[linewidth=0.2,arrowsize=1.5]{->}(0,2)(2,5)(3,11.5)
\end{pspicture}
\begin{pspicture}(-7.5,-2.5)(17,13)
\pscircle[linewidth=0.2,linestyle=dotted](0,5){7.5}
\psline[linewidth=0.2,arrowsize=1.5]{->}(-3,11.5)(-2.5,9)
\pscurve[linewidth=0.2](-2.5,9)(-2.3,8)(-1.5,6)(-0.5,5)
\psline[linewidth=0.2,arrowsize=1.5]{->}(0,12.5)(0.5,10)
\pscurve[linewidth=0.2](0.5,10)(0.6,9.5)(0.4,7)(-0.5,5)
\psline[linewidth=0.2,arrowsize=1.5]{->}(-0.5,5)(3,-1.5)
\pscurve[linewidth=0.2,border=0.6,arrowsize=1.5]{->}(3,11.5)(1.5,1.5)(-3,-1.5)
\rput(13,5){$\overset{\text{R}\mathrm{VI}^z}{\longleftrightarrow}$}
\end{pspicture}
\begin{pspicture}(-7.5,-2.5)(15,13)
\pscircle[linewidth=0.2,linestyle=dotted](0,5){7.5}
\psline[linewidth=0.2,arrowsize=1.5]{->}(-3,11.5)(-2.8,8.5)
\pscurve[linewidth=0.2](-2.8,8.5)(-2,4)(1.5,1.5)
\psline[linewidth=0.2,arrowsize=1.5]{->}(0,12.5)(0.5,10)
\pscurve[linewidth=0.2](0.5,10)(1.2,6)(1.5,1.5)
\psline[linewidth=0.2,arrowsize=1.5]{->}(1.5,1.5)(3,-1.5)
\pscurve[linewidth=0.2,border=0.6,arrowsize=1.5]{->}(3,11.5)(-1,6)(-2,4)(-3,-1.5)
\end{pspicture}
\begin{pspicture}(-7.5,-2.5)(17,13)
\pscircle[linewidth=0.2,linestyle=dotted](0,5){7.5}
\pscurve[linewidth=0.2,arrowsize=1.5]{->}(-0.5,5)(0.3,9)(0,12.5)
\pscurve[linewidth=0.2,arrowsize=1.5]{->}(-0.5,5)(-2.5,8)(-3,11.5)
\psline[linewidth=0.2,arrowsize=1.5]{->}(3,-1.5)(-0.5,5)
\pscurve[linewidth=0.2,border=0.6,arrowsize=1.5]{->}(-3,-1.5)(1.5,1.5)(3,11.5)
\rput(13,5){$\overset{\text{R}\mathrm{VI}^u}{\longleftrightarrow}$}
\end{pspicture}
\begin{pspicture}(-7.5,-2.5)(15,13)
\pscircle[linewidth=0.2,linestyle=dotted](0,5){7.5}
\pscurve[linewidth=0.2,arrowsize=1.5]{->}(1.5,1.5)(1,7)(0,12.5)
\pscurve[linewidth=0.2,arrowsize=1.5]{->}(1.5,1.5)(-2,4)(-3,11.5)
\psline[linewidth=0.2,arrowsize=1.5]{->}(3,-1.5)(1.5,1.5)
\pscurve[linewidth=0.2,border=0.6,arrowsize=1.5]{->}(-3,-1.5)(-2,4)(-1,6)(3,11.5)
\end{pspicture}
\medskip

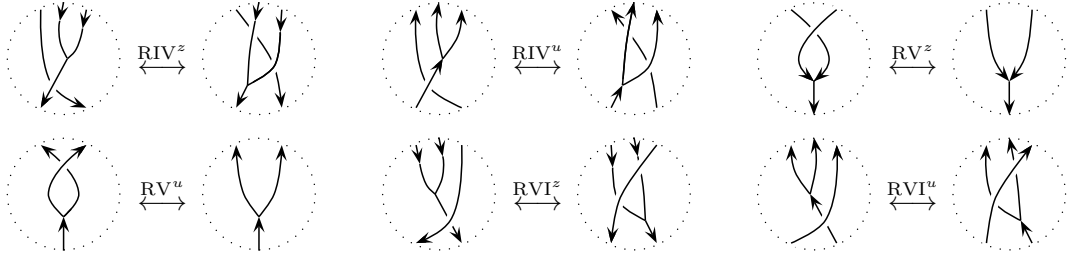
\captionof{figure}{Reidemeister moves for well-oriented graph diagrams} \label{pic:RMovesForWellOrGraphs}
\end{center}

Superscripts~$^z$ and~$^u$ refer to the zip or unzip vertex involved in the move.

\begin{proof}
Move R$\mathrm{IV}^u$ for another orientation is treated in Figure~\ref{pic:RMovesForWellOrGraphsProof}; an alternative orientation of R$\mathrm{V}^u$ is dealt with in Figure~\ref{pic:RMovesForWellOrGraphsProof2}. Other moves and orientations can be treated in a similar way.

\begin{center}
\begin{pspicture}(-5,-2.5)(12,13)
\pscurve[linewidth=0.2,arrowsize=1.5]{->}(-3,12.5)(-1.5,1.5)(3,-1.5)
\pscurve[linewidth=0.2,arrowsize=1.5]{->}(0.5,5)(-0.3,9)(0,12.5)
\pscurve[linewidth=0.2,arrowsize=1.5]{->}(0.5,5)(2.5,8)(3,12.5)
\psline[linewidth=0.2,border=0.6,arrowsize=1.5]{->}(-3,-1.5)(0.5,5)
\rput(10,5){$\overset{\text{R}\mathrm{II}}{\longleftrightarrow}$}
\end{pspicture}
\begin{pspicture}(-7,-2.5)(15,18)
\pscurve[linewidth=0.2,arrowsize=1.5]{->}(-3,12.5)(4.5,6)(-2.5,9)(-1.5,1.5)(3,-1.5)
\pscurve[linewidth=0.2,border=0.6,arrowsize=1.5]{->}(0.5,5)(-0.3,9)(0,12.5)
\pscurve[linewidth=0.2,border=0.6,arrowsize=1.5]{->}(0.5,5)(2.5,8)(3,12.5)
\psline[linewidth=0.2,border=0.6,arrowsize=1.5]{->}(-3,-1.5)(0.5,5)
\pscurve[linewidth=0.2,arrowsize=1.5]{->}(0.5,5)(-0.3,9)(0,12.5)
\pscurve[linewidth=0.2,arrowsize=1.5]{->}(0.5,5)(2.5,8)(3,12.5)
\rput(13,5){$\overset{\text{R}\mathrm{IV}^u}{\longleftrightarrow}$}
\end{pspicture}
\begin{pspicture}(-7,-2.5)(15,18)
\pscurve[linewidth=0.2,arrowsize=1.5]{->}(-3,12.5)(2.5,8)(4.5,2)(-3.5,2)(-2,1)(3,-1.5)
\pscurve[linewidth=0.2,border=0.6,arrowsize=1.5]{->}(0.5,5)(-0.3,9)(0,12.5)
\pscurve[linewidth=0.2,border=0.6,arrowsize=1.5]{->}(0.5,5)(2.5,8)(3,12.5)
\psline[linewidth=0.2,border=0.6,arrowsize=1.5]{->}(-3,-1.5)(0.5,5)
\pscurve[linewidth=0.2,arrowsize=1.5]{->}(0.5,5)(-0.3,9)(0,12.5)
\pscurve[linewidth=0.2,arrowsize=1.5]{->}(0.5,5)(2.5,8)(3,12.5)
\rput(13,5){$\overset{\text{R}\mathrm{II}}{\longleftrightarrow}$}
\end{pspicture}
\begin{pspicture}(-7,-2.5)(12,7)
\pscurve[linewidth=0.2,arrowsize=1.5]{->}(-3,12.5)(1,6)(2,4)(3,-1.5)
\pscurve[linewidth=0.2,border=0.6,arrowsize=1.5]{->}(-1.5,1.5)(-1,7)(0,12.5)
\pscurve[linewidth=0.2,border=0.6,arrowsize=1.5]{->}(-1.5,1.5)(2,4)(3,12.5)
\pscurve[linewidth=0.2](-1.5,1.5)(-1,7)(0,12.5)
\psline[linewidth=0.2,arrowsize=1.5]{->}(-3,-1.5)(-1.5,1.5)
\end{pspicture}

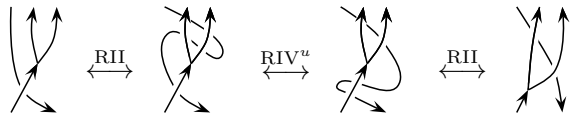
\captionof{figure}{Reidemeister move~$\mathrm{IV}^u$ for another orientation} \label{pic:RMovesForWellOrGraphsProof}
\end{center}

\begin{center}
\begin{pspicture}(-7,-2.5)(13,13)
\pscurve[linewidth=0.2,arrowsize=1.5]{->}(0,2)(2,5)(-3,12.5)
\pscurve[linewidth=0.2,border=0.6,arrowsize=1.5]{->}(3,12.5)(-2,5)(0,2)
\psline[linewidth=0.2,arrowsize=1.5]{->}(0,2)(0,-2.5)
\psline[linewidth=0.2](0,2)(1,3)
\rput(9,5){$\overset{\text{R}\mathrm{IV}^u}{\longleftrightarrow}$}
\end{pspicture}
\begin{pspicture}(-7,-2.5)(17,13)
\pscurve[linewidth=0.2,arrowsize=1.5]{->}(0,7)(2,6)(3.5,2)(5,4)(-2,6)(-3,12.5)
\psline[linewidth=0.2,border=0.6,arrowsize=1.5]{->}(0,7)(0,-2.5)
\pscurve[linewidth=0.2,border=0.6](0,7)(2,6)(3.5,2)(5,4)
\psline[linewidth=0.2](0,7)(0,-2.5)
\pscurve[linewidth=0.2,arrowsize=1.5]{->}(3,12.5)(-2,10)(0,7)
\rput(13,5){$\overset{\text{R}\mathrm{I}}{\longleftrightarrow}$}
\end{pspicture}
\begin{pspicture}(-7,-2.5)(16,13)
\pscurve[linewidth=0.2,arrowsize=1.5]{->}(0,7)(2,6)(3.5,4)(2,2)(-2,6)(-3,12.5)
\psline[linewidth=0.2,border=0.6,arrowsize=1.5]{->}(0,7)(0,-2.5)
\psline[linewidth=0.2](0,7)(1,6.5)
\psline[linewidth=0.2](0,7)(0,-2.5)
\psline[linewidth=0.2,arrowsize=1.5]{->}(3,12.5)(0,7)
\rput(12,5){$\overset{\text{R}\mathrm{V}^u}{\longleftrightarrow}$}
\end{pspicture}
\begin{pspicture}(-7,-2.5)(16,13)
\pscurve[linewidth=0.2,arrowsize=1.5]{->}(0,7)(-2,5)(3,1.5)(-1.5,1)(-3.5,5)(-3,12.5)
\psline[linewidth=0.2,border=0.6,arrowsize=1.5]{->}(0,7)(0,-2.5)
\psline[linewidth=0.2](0,7)(-1,6.3)
\psline[linewidth=0.2](0,7)(0,-2.5)
\psline[linewidth=0.2,arrowsize=1.5]{->}(3,12.5)(0,7)
\rput(12,5){$\overset{\text{R}\mathrm{II}}{\longleftrightarrow}$}
\end{pspicture}
\begin{pspicture}(-7.5,-2.5)(15,13)
\psline[linewidth=0.2,arrowsize=1.5]{->}(0,2)(0,-2.5)
\pscurve[linewidth=0.2,arrowsize=1.5]{->}(0,2)(-2,5)(-3,12.5)
\pscurve[linewidth=0.2,arrowsize=1.5]{->}(3,12.5)(2,5)(0,2)
\end{pspicture}

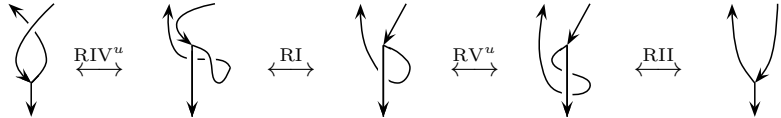
\captionof{figure}{Reidemeister move~$\mathrm{V}^u$ for another orientation} \label{pic:RMovesForWellOrGraphsProof2}
\end{center}
\end{proof}

Although our orientation restriction prevents one from working with
arbitrary oriented graphs, unoriented graphs can be dealt with
thanks to the following observation:

\begin{proposition}\label{P:UnVsWellOrient}
Any abstract or knotted $3$-valent graph can be well-oriented.
\end{proposition}

\begin{proof}
Take an abstract unoriented graph~$\Gamma$. Suppose all its vertices to be of odd valency. We call a \textit{path} a sequence of its pairwise distinct edges $e_1, \ldots, e_k$, the endpoints $(s_i,t_i)$ of each~$e_i$ being ordered,  such that $t_i$ and $s_{i+1}$ coincide for each $i$. 
Choose a \textit{maximal path}~$\gamma$ in~$\Gamma$ --- i.e., a path which is not a sub-path of a longer one. Deleting~$\gamma$ from~$\Gamma$ and forgetting all the isolated vertices possibly formed after that, one gets a graph~$\Gamma\setminus\gamma$, whose vertices are still of odd valency. Indeed, the valency subtracted from internal vertices of~$\gamma$ is even (since we enter and leave them the same number of times); the same argument works for the first and last vertices if they coincide (in which case we call them internal as well); if they are distinct, then their full valencies are subtracted --- otherwise~$\gamma$ could be prolonged, which would contradict its maximality --- and so they are discarded. Now let~$\Gamma$ be $3$-valent. Iterating the argument above, one presents~$\Gamma$ as a disjoint union of paths, each vertex occurring in at most two paths and being internal for the first path it belongs to. Orienting each edge~$e_i$ in each path from~$s_i$ to~$t_i$, one well-orients~$\Gamma$.  
\end{proof}

Thus, in order to compare two unoriented graphs, it is sufficient to
compare the sets of their well-oriented versions.

\subsection*{A new coloring approach via qualgebras}

Now, for well-oriented graph diagrams, consider coloring rule from Figure~\ref{pic:Colorings}\rcircled{B}, where~$\op$ is another binary operation on the quandle $(\Q,\lhd)$. Trying to render these rules topological, one arrives to the notion of qualgebra, central to this paper.

\begin{definition}\label{D:Qualgebra}
A set~$\Q$ endowed with two binary operations~$\lhd$ and~$\op$ is called a \emph{qualgebra} if it satisfies Axioms \eqref{E:SD}-\eqref{E:QAComm} (see page~\pageref{E:SD}).
\end{definition}

 The term ``qualgebra'' comes from terms ``quandle'' and ``algebra'' zipped together, as shown on Figure~\ref{pic:Term}. It underlines the presence of two interacting operations in this structure.

 Algebraically, this definition can be restated in a more structural way. Namely, consider a set~$\Q$ endowed with two binary operations~$\lhd$ and~$\op$, and define an operator 
\begin{align*}
\sigma_\lhd: \Q \times \Q &\longrightarrow \Q \times \Q,\\
(a,b) & \longmapsto (b,a \lhd b).
\end{align*} 
Then $(\Q,\lhd,\op)$ is a qualgebra if and only if $(\Q,\sigma_\lhd,\op)$ is a \textit{braided algebra} which is \textit{braided-commutative} but not necessarily associative, and such that the Yang-Baxter operator~$\sigma_\lhd$ preserves the diagonal of~$\Q$.
 
 Remark that Axiom~\eqref{E:SD} could be omitted from the definition, as it is a consequence of~\eqref{E:QA1} and~\eqref{E:QAComm}:  
$$(a\lhd b)\lhd c \overset{\eqref{E:QA1}}{=} a\lhd (b\op c) \overset{\eqref{E:QAComm}}{=} a\lhd (c\op (b \lhd c)) \overset{\eqref{E:QA1}}{=} (a\lhd c)\lhd (b \lhd c);$$
we will include or omit this axiom according to our needs.

For further reference, let us also note the compatibility relations between operations~$\op$ and~$\wlhd$.

\begin{lemma}\label{L:QualgebraProperties}
A qualgebra $(\Q,\lhd,\op)$ enjoys the following properties:
\begin{align}
&a\wlhd (b\op c) = (a\wlhd c)\wlhd b,\label{E:QA1'}\\
&(a\op b)\wlhd c = (a\wlhd c) \op (b\wlhd c),\label{E:QA2'}\\
&(a \wlhd b)\op b = b \op a.\label{E:QAComm'}
\end{align}
\end{lemma}

\begin{proof}
Let us show~\eqref{E:QA1'}, the proof for the remaining relations being similar. Applying~\eqref{E:QA1} to elements $a \wlhd (b\op c)$, $b$ and $c$, one gets
$$(a \wlhd (b \op c)) \lhd (b \op c) = ((a \wlhd (b \op c)) \lhd b) \lhd c.$$
The left-hand side equals~$a$ because of~\eqref{E:Inv}. Now, apply the map $x \mapsto (x \wlhd c) \wlhd b$ to both sides:
$$(a \wlhd c) \wlhd b = ((((a \wlhd (b \op c)) \lhd b) \lhd c) \wlhd c) \wlhd b.$$
Using~\eqref{E:Inv} for the right-hand side this time, one obtains~\eqref{E:QA1'}.
\end{proof}

\medskip

Now, returning to colorings of graphs, one gets

\begin{proposition}\label{P:QualgebraColor}
Take a set~$\Q$ endowed with two binary operations~$\lhd$ and~$\op$. Coloring rules from Figure~\ref{pic:Colorings}\rcircled{A}\&\rcircled{B} are topological if and only if $(\Q,\lhd,\op)$ is a qualgebra.
\end{proposition}

\begin{proof}
The equivalence between the compatibility of the coloring rule~\ref{pic:Colorings}\rcircled{A} with Reidemeister moves $\mathrm{I}$-$\mathrm{III}$ on the one hand, and Axioms \eqref{E:SD}-\eqref{E:Idem} on the other hand, was discussed in Example~\ref{EX:QuandleCol}. Let us turn to the remaining three moves, with orientations from Lemma~\ref{L:RMovesForWellOrGraphs}. Analyzing move R$\mathrm{IV}^z$ (Figure~\ref{pic:Qualgebras}), one notices that on each side the three colors on the top completely determine all the remaining colors, in particular the colors on the bottom. Then, the coloring bijection from Definition~\ref{D:TopColRules} takes place if and only if the induced bottom colors coincide on the two sides, which is equivalent to Axiom~\eqref{E:QA1}. An analogous argument shows that for move R$\mathrm{IV}^u$, the coloring bijection is equivalent to Axiom~\eqref{E:QA1'}, which, in the presence of~\eqref{E:Inv}, is the same as~\eqref{E:QA1} (cf. the proof of Lemma~\ref{L:QualgebraProperties}). Similarly, one checks that for both the zip and unzip versions of R$\mathrm{VI}$ (respectively, R$\mathrm{V}$) the coloring bijection is equivalent to Axiom~\eqref{E:QA2} (respectively, \eqref{E:QAComm}). 

\begin{center}
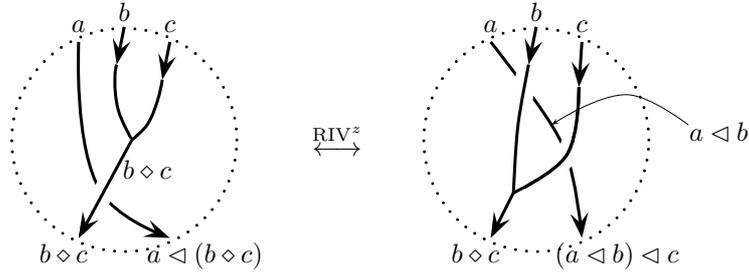
 
\psset{unit=2mm}
\begin{pspicture}(-7.5,-3.5)(19,13.5)
\pscircle[linewidth=0.2,linestyle=dotted](0,5){7.5}
\pscurve[linewidth=0.2,arrowsize=1.2]{->}(-3,11.5)(-1.5,1.5)(3,-1.5)
\psline[linewidth=0.2,arrowsize=1.2]{->}(3,11.5)(2.5,9)
\pscurve[linewidth=0.2](2.5,9)(2.3,8)(1.5,6)(0.5,5)
\psline[linewidth=0.2,arrowsize=1.2]{->}(0,12.5)(-0.5,10)
\pscurve[linewidth=0.2](-0.5,10)(-0.6,9.5)(-0.4,7)(0.5,5)
\psline[linewidth=0.2,border=0.6,arrowsize=1.2]{->}(0.5,5)(-3,-1.5)
\rput(-3,12.5){$a$}
\rput(0,13.5){$b$}
\rput(3,12.5){$c$}
\rput(-4,-2.5){$b \op c$}
\rput(1.5,3){$b \op c$}
\rput(5.3,-2.7){$a \lhd (b \op c)$}
\rput(14,5){$\overset{\text{R}\mathrm{IV}^z}{\longleftrightarrow}$}
\end{pspicture}
\begin{pspicture}(-7.5,-3.5)(15,13.5)
\pscircle[linewidth=0.2,linestyle=dotted](0,5){7.5}
\pscurve[linewidth=0.2,arrowsize=1.2]{->}(-3,11.5)(1,6)(2,4)(3,-1.5)
\psline[linewidth=0.2,arrowsize=1.2]{->}(3,11.5)(2.8,8.5)
\pscurve[linewidth=0.2,border=0.6](2.8,8.5)(2,4)(-1.5,1.5)
\psline[linewidth=0.2,arrowsize=1.2]{->}(0,12.5)(-0.5,10)
\pscurve[linewidth=0.2,border=0.6](-0.5,10)(-1.2,6)(-1.5,1.5)
\pscurve[linewidth=0.2](2.8,8.5)(2,4)(-1.5,1.5)
\psline[linewidth=0.2,arrowsize=1.2]{->}(-1.5,1.5)(-3,-1.5)
\rput(-3,12.5){$a$}
\rput(0,13.5){$b$}
\rput(3,12.5){$c$}
\rput(-4,-2.5){$b \op c$}
\rput(5.3,-2.7){$(a \lhd b) \lhd c$}
\rput(12,5.5){$a \lhd b$}
\pscurve[linewidth=0.05]{->}(10,6)(6,8)(1,6)
\end{pspicture}
\psset{unit=1mm}
\captionof{figure}{Qualgebra axioms via coloring rules for graph diagrams}\label{pic:Qualgebras}\qedhere
\end{center}
\end{proof}

\begin{remark}
Certainly, we could have used different operations~$\op_z$ and~$\op_u$ for coloring rules around zip and unzip vertices. However, our simplified choice already produces powerful invariants; moreover, it is natural if one thinks in terms of generalizations of (multiple) conjugation quandle colorings of graphs.
\end{remark}

Lemma~\ref{L:CountInvar} now allows one to construct qualgebra coloring invariants for graphs:
 
\begin{corollary}\label{L:QualgebraCountInvar}
Take a qualgebra $(\Q,\lhd,\op)$ and consider $\Q$-coloring rules from Figure~\ref{pic:Colorings}\rcircled{A}\&\rcircled{B}. The (possibly infinite) quantity $\#\Col_{\Q}(\D)$ does not depend on the choice of a diagram~$\D$ representing a well-oriented $3$-valent knotted graph~$\Gamma$.
\end{corollary}

\begin{proof}
Proposition~\ref{P:QualgebraColor} guarantees that the coloring rules in question are topological.
Lemma~\ref{L:CountInvar} then tells that the function $\D \mapsto \#\Col_{\Q}(\D)$ is well-defined on R-equivalence classes of diagrams, which, according to \cite{KauffmanGraphs,YamadaGraphs,YetterGraphs}, correspond to isotopy classes of graphs.
\end{proof}

One thus gets a systematic way of producing invariants of well-oriented (or unoriented, cf. Proposition~\ref{P:UnVsWellOrient}) graphs.

\subsection*{Group qualgebras}

We now show that groups are an important source of qualgebras, playing also a significant motivational role.

\begin{example}\label{EX:GroupQualgebra}
A conjugation quandle together with the group multiplication operation $a \op b = ab$ is a qualgebra, called a \textit{group qualgebra}; a direct verification of all the axioms is easy. For this qualgebra, the coloring rule from Figure~\ref{pic:Colorings}\rcircled{B} repeats that from Figure~\ref{pic:QuandlesForGraphs}\rcircled{B}. Thus our qualgebra coloring rules and resulting graph invariants generalize the group coloring rules and corresponding invariants.
\end{example}

While from the topological perspective quandle axioms \eqref{E:SD}-\eqref{E:Idem} can be viewed as algebraic incarnations of Reidemeister moves for knots, from the algebraic viewpoint they are often interpreted as an axiomatization of the conjugation operation in a group. Concretely, if a relation involving only conjugation holds in every group, then it can be deduced from the quandle axioms (cf. \cite{Joyce,Dehornoy2}). In a similar way, as shown in the (proof of) Proposition~\ref{P:QualgebraColor}, topologically additional qualgebra axioms \eqref{E:QA1}-\eqref{E:QAComm} can be regarded as algebraic incarnations of specific R-moves for $3$-valent graphs. Algebraically, they encode major relations between conjugation and multiplication operations in a group (cf. Table~\ref{tab:3levels}). However, we shall see below that not all the conjugation/multiplication relations are captured by the qualgebra structure.

\begin{center}
\begin{tabular}{|c|c|c|}
\hline
\textbf{abstract level} & quandle axioms & specific qualgebra axioms \\
\hline
\textbf{group level} & conjugation & conjugation/multiplication interaction \\
\hline
\textbf{topological level} & moves R$\mathrm{I}$-R$\mathrm{III}$ & moves R$\mathrm{IV}$-R$\mathrm{VI}$ \\
\hline
\end{tabular}
\captionof{table}{Different viewpoints on quandles and qualgebras} \label{tab:3levels}
\end{center}

A slight variation of Example~\ref{EX:GroupQualgebra} is first due:

\begin{example}\label{EX:SubGroupQualgebra}
New examples of qualgebras can be derived by considering sub-qualgebras of given qualgebras. In the case of group qualgebras, these are simply subsets closed under conjugation and multiplication operations, but not necessarily under taking inverse. For instance, positive integers~$\NN$ form a sub-qualgebra of the group qualgebra of~$\ZZ$.
\end{example}

Note that sub-qualgebras of group qualgebras do not necessarily contain the neutral element or inverses. However, they clearly remain associative:

\begin{definition}\label{D:AssQualgebra}
A qualgebra $(\Q,\lhd,\op)$ is called \emph{associative} if the operation~$\op$ is such, i.e., if for all elements of~$\Q$ one has
\begin{equation}\label{E:Ass}
(a \op b) \op c = a \op (b \op c).
\end{equation}
\end{definition}

Examples of non-associative qualgebras will be given in Section~\ref{sec:Qualgebras4}.

Recall that in the quandle setting, the free quandle on a set~$\SS$ can be seen as the $\SS$-generated sub-quandle of the conjugation quandle of the free group on~$\SS$. This explains the fundamental role of conjugation quandles among all quandles. One would expect a similar result in the associative qualgebra setting (the necessity to impose the associativity is explained above). However, this is false:

\begin{proposition}\label{P:QualgebraAxiomatize}
Take a set~$\SS$ with at least $2$ elements. Consider the map from the free associative qualgebra $FAQA_{\SS}$ on~$\SS$ to the group qualgebra of the free group $FG_{\SS}$ on~$\SS$, sending every $a \in \SS$ to itself. This map is not injective.
\end{proposition}

The proof of this result is slightly technical and is therefore presented in Appendix~\ref{A:Proof}.

\subsection*{Related constructions and ``non-qualgebrizable'' quandles}

Group qualgebras and their sub-qualgebras are far from covering all examples of qualgebra structure. We have just seen a manifestation of this fact: Relation~\eqref{E:FreeQA}, even though automatic in group qualgebras, fails in some other associative qualgebras. Moreover, in Section~\ref{sec:Qualgebras4} we shall show that even in small size there are some exotic qualgebras exhibiting very ``non-group-like'' properties: they are neither cancellative, nor associative, nor unital. Our choice of qualgebra axioms, resulting in the structure's richness (illustrated in particular by such exotic examples), was dictated by the desired applications to graph invariants. Here we mention some related structures from the literature, appearing in different frameworks and exhibiting  dissimilar properties. 

First, observe that the associativity, absent from our topological picture, does become relevant when one works with \textit{handlebody-knots} (cf. \cite{IshiiHKnots}). In particular it appears, together with Axioms \eqref{E:SD}, \eqref{E:Inv}, \eqref{E:QA1} and \eqref{E:QA2}, in A.Ishii's definition of \textit{multiple conjugation quandle}, the latter being tailored for producing handlebody-knot invariants. Remark that algebraically, MCQs inherit many properties of groups, since they are formed by gluing several groups together.

Besides the topological and algebraic settings described above, Axioms \eqref{E:QA1}-\eqref{E:QAComm} also emerge in a completely different set-theoretical context. Namely, together with the associativity of~$\op$ and the existence of a neutral element~$1$ for~$\op$ satisfying moreover $1 \lhd a = 1$ and $a \lhd 1 = a$ for all $a \in \Q$, they define a \textit{(right-)distributive monoid} (or, in other sources, RD algebra). The examples of elementary embeddings, Laver tables and extended braids, all of which admit rich distributive monoid structures, have motivated an extensive study of the concept (cf. for instance \cite{DehornoyLDM,DrapalLDM,DrapalLDM2,DehornoyBinf}, or Chapter~$\mathrm{XI}$ of~\cite{Dehornoy2} for a comprehensive exposition). A weaker \textit{augmented (right-)distributive system} structure of P.Dehornoy obeys only three axioms: \eqref{E:SD}, \eqref{E:QA1}, and \eqref{E:QA2}; the major example here is that of parenthesized braids (cf. \cite{DehornoyParBr,DehornoyFreeALDS}). Our qualgebras are particular cases of augmented distributive systems.

\medskip
We finish with some remarks concerning the relations between quandle and qualgebra structures. Any quandle can be embedded (as a sub-quandle) into a qualgebra (cf. \cite{LebedQAAlg}). Further, some quandles can be upgraded to qualgebras using several different operations~$\op$ (cf. Section~\ref{sec:Qualgebras4} for examples). Here we give an example of a family of quandles  which can not be turned into qualgebras, and of a quandle admitting exactly one compatible operation~$\op$.

\begin{example}\label{EX:NonQualgebras}
A \textit{dihedral quandle} is the set $\ZZ/n\ZZ$ endowed with the operation $a \lhd b = 2b -a$ $(\operatorname{mod} n)$. Suppose that $\ZZ/n\ZZ$ can be endowed with an additional operation~$\op$ satisfying \eqref{E:QA1}. Then for all $a,b,c \in \ZZ/n\ZZ$, the element $(a\lhd b)\lhd c = 2c-2b +a$ would coincide with $a\lhd (b\op c) = 2(b\op c) -a$, thus $2a = 2(b\op c)-2c+2b$ would not depend on~$a$, which is impossible if $n \neq 2$.
\end{example}

\begin{example}\label{EX:UniqueQualgebras}
Consider the conjugation quandle of the symmetric group~$S_3$. As usual, operation $a \op b = ab$ turns it into a group quandle. Let us show that this is the only qualgebrization of this quandle. Indeed, Axiom~\eqref{E:QA1} imposes the values of $(12) \lhd (a \op b)$ and $(123) \lhd (a \op b)$ for all $a,b \in S_3$; it remains to show that the values $(12) \lhd x$ and $(123) \lhd x$ uniquely identify an $x\in S_3$. This follows by direct computations:
\begin{align*}
(12) \lhd x &= \begin{cases} (12) \text{ if } x \in \{\Id, (12)\}, \\
(23) \text{ if } x \in \{ (132), (13)\}, \\
(13) \text{ if } x \in \{ (123), (23)\}; \end{cases} &
(123) \lhd x &= \begin{cases} (123) \text{ if } x \in \{\Id, (123), (132)\}, \\
(213) \text{ if } x \in \{ (12), (23), (13)\}. \end{cases} 
\end{align*}
\end{example}

\section{Isosceles colorings and squandles}\label{sec:Isosceles}

In concrete situations, one sometimes has to deal with pairs of graphs for which the $\Q$-coloring counting invariants  from Corollary~\ref{L:QualgebraCountInvar} coincide for certain qualgebras~$\Q$, but which can be distinguished if only a particular kind of colorings is taken into account. After a short survey of the development of such ``special coloring'' ideas in the literature, we introduce a particular kind of qualgebra colorings, allowing one to distinguish, for instance, the two theta-curves from Figure~\ref{pic:theta}. 

\subsection*{Special colorings}

Start with group coloring rules for arbitrary oriented graphs (Figures~\ref{pic:Colorings}\rcircled{A} and~\ref{pic:QuandlesForGraphs}\rcircled{B}). The most natural particular kind of corresponding colorings is the one where the colors of arcs adjacent to the same vertex coincide, up to orientations. This means using the coloring rule from Figure~\ref{pic:SpecialCol}\rcircled{A}, where color~$a$ should be chosen for arcs oriented from the vertex, and color~$a^{-1}$ for the remaining ones. Such colorings can be traced back to C.Livingston's 1995 study of vertex constant graph groups (\cite{Livingston}). These ideas were generalized in 2007 by T.Fleming and B.Mellor (\cite{VirtualGraphs}) to the case of \textit{symmetric quandle}. The latter is a quandle~$\Q$ endowed with a \textit{good involution}, i.e., a map $\rho:\Q \rightarrow \Q$ satisfying, for all elements of~$\Q$, 
\begin{align}
& \rho(\rho(a)) = a,\label{E:SymQu1}\\
& \rho(a) \lhd b = \rho(a \lhd b),\label{E:SymQu2}\\
& a \lhd \rho(b) = a \wlhd b.\label{E:SymQu3}
\end{align}
Symmetric quandles were defined by S.Kamada in \cite{KamadaSymm}. The basic example is our favourite conjugation quandle, with $\rho(a)= a^{-1}$. Now, for a symmetric quandle~$\Q$, Fleming-Mellor's coloring rule for graphs is presented on Figure~\ref{pic:SpecialCol}\rcircled{B}; notations $a^{+1}=a$, $a^{-1}=\rho(a)$ are used here, and the choice in~$\pm 1$ is controlled by the same rule as for group colorings. This rule generalizes that from Figure~\ref{pic:SpecialCol}\rcircled{A}, and corresponding colorings can be seen as special among the quandle colorings in the sense of~\ref{pic:QuandlesForGraphs}\rcircled{C}. To see that one gets topological coloring rules, it suffices to check that a special coloring remains such after an R-move and the corresponding coloring change, which is done by an easy direct verification (cf. the proof of Proposition~\ref{P:SpecialQualgebraColor}). M.Niebrzydowski further generalized these ideas to an arbitrary quandle case (see \cite{NiebrzydowskiGraphs}).

\begin{center}
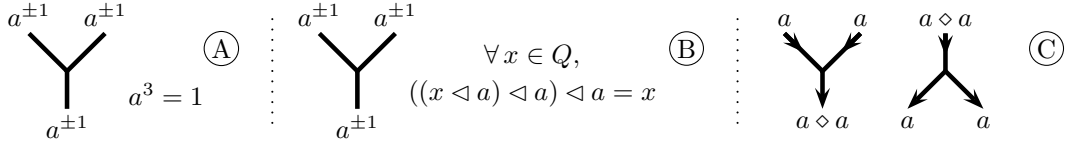

\begin{pspicture}(37,15)
\psline[linewidth=0.6](0,10)(5,5)(5,0)
\psline[linewidth=0.6](10,10)(5,5)
\rput[b](0,11){$a^{\pm 1}$}
\rput[b](10,11){$a^{\pm 1}$}
\rput[t](5,-1){$a^{\pm 1}$}
\rput(18,2){$a^3 = 1$}
\rput(25,8){\rcircled{A}}
\psline[linestyle=dotted](32,-2)(32,12)
\end{pspicture}
\begin{pspicture}(60,15)
\psline[linewidth=0.6](0,10)(5,5)(5,0)
\psline[linewidth=0.6](10,10)(5,5)
\rput[b](0,11){$a^{\pm 1}$}
\rput[b](10,11){$a^{\pm 1}$}
\rput[t](5,-1){$a^{\pm 1}$}
\rput(28,7){$\forall \, x \in \Q,$}
\rput(28,2){$((x \lhd a) \lhd a) \lhd a =x$}
\rput(48,8){\rcircled{B}}
\psline[linestyle=dotted](55,-2)(55,12)
\end{pspicture}
\begin{pspicture}(15,15)
\psline[linewidth=0.6,arrowsize=2]{->}(0,10)(2.5,7.5)
\psline[linewidth=0.6,arrowsize=2]{->}(0,10)(5,5)(5,0)
\psline[linewidth=0.6,arrowsize=2]{->}(10,10)(7.5,7.5)
\psline[linewidth=0.6](10,10)(5,5)
\rput[b](0,11){$a$}
\rput[b](10,11){$a$}
\rput[t](5,-1){$a \op a$}
\end{pspicture}
\begin{pspicture}(22,15)
\psline[linewidth=0.6,arrowsize=2]{->}(5,10)(5,7)
\psline[linewidth=0.6,arrowsize=2]{->}(5,10)(5,5)(0,0)
\psline[linewidth=0.6,arrowsize=2]{->}(5,5)(10,0)
\rput[t](0,-1){$a$}
\rput[t](10,-1){$a$}
\rput[b](5,11){$a \op a$}
\rput(18,8){\rcircled{C}}
\end{pspicture}
\medskip

\captionof{figure}{Examples of special coloring} \label{pic:SpecialCol}
\end{center}

\subsection*{Isosceles colorings}

We now return to qualgebra colorings for well-oriented graphs. The class of special colorings we propose to study here is the following:

\begin{definition}\label{D:IsoscelesColor}
Take a qualgebra~$(\Q,\lhd,\op)$ and a $\Q$-colored well-oriented graph diagram~$(\D,\C)$. A $3$-valent vertex of~$\D$ is called \textit{$\C$-isosceles} if~$\C$ assigns the same colors to its two adjacent co-oriented arcs. The coloring~$\C$ itself is called \emph{isosceles} if all vertices of~$\D$ are $\C$-isosceles.
\end{definition}

In other words, working with isosceles colorings means considering coloring rule~\ref{pic:SpecialCol}\rcircled{C}.

\begin{proposition}\label{P:SpecialQualgebraColor}
Given a qualgebra~$(\Q,\lhd,\op)$, the coloring rules from Figures~\ref{pic:Colorings}\rcircled{A} and~\ref{pic:SpecialCol}\rcircled{C} are topological.
\end{proposition}

\begin{proof}
Since isosceles colorings are particular instances of those from Proposition~\ref{P:QualgebraColor}, which are controlled by topological rules, it suffices to check that an isosceles coloring remains such after an R-move and the corresponding coloring change. For moves R$\mathrm{I}$-R$\mathrm{III}$ and R$\mathrm{V}$ it is obvious, since they do not change the colors around isosceles trivalent vertices. Move R$\mathrm{VI}^u$ is treated on Figure~\ref{pic:RVIu}: the top three colors determine all the remaining ones (note that the bottom colors coincide due to~\eqref{E:QA2'}), and for any of the two diagrams being isosceles means satisfying $a=b$ (since the map $x \mapsto x \wlhd c$ is a bijection on~$\Q$). Moves  R$\mathrm{VI}^z$ and R$\mathrm{IV}$ are treated similarly.

\begin{center}
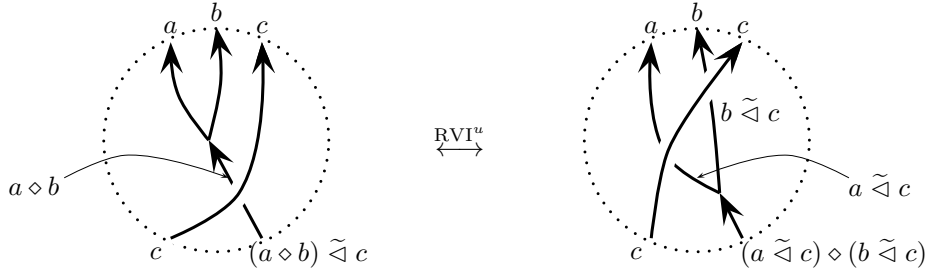

\psset{unit=2mm}
\begin{pspicture}(-15,-3.5)(23,13.5)
\pscircle[linewidth=0.2,linestyle=dotted](0,5){7.5}
\pscurve[linewidth=0.2,arrowsize=1.5]{->}(-0.5,5)(0.3,9)(0,12.5)
\pscurve[linewidth=0.2,arrowsize=1.5]{->}(-0.5,5)(-2.5,8)(-3,11.5)
\psline[linewidth=0.2,arrowsize=1.5]{->}(3,-1.5)(-0.5,5)
\pscurve[linewidth=0.2,border=0.6,arrowsize=1.5]{->}(-3,-1.5)(1.5,1.5)(3,11.5)
\rput(-3,12.5){$a$}
\rput(0,13.5){$b$}
\rput(3,12.5){$c$}
\rput(-4,-2.5){$c$}
\rput(-12,2){$a \op b$}
\pscurve[linewidth=0.05]{->}(-10,2.5)(-6,4)(0.8,2.5)
\rput(6,-2.5){$(a \op b) \wlhd c$}
\rput(16,5){$\overset{\text{R}\mathrm{VI}^u}{\longleftrightarrow}$}
\end{pspicture}
\begin{pspicture}(-8,-3.5)(18,13.5)
\pscircle[linewidth=0.2,linestyle=dotted](0,5){7.5}
\pscurve[linewidth=0.2,arrowsize=1.5]{->}(1.5,1.5)(1,7)(0,12.5)
\pscurve[linewidth=0.2,arrowsize=1.5]{->}(1.5,1.5)(-2,4)(-3,11.5)
\psline[linewidth=0.2,arrowsize=1.5]{->}(3,-1.5)(1.5,1.5)
\pscurve[linewidth=0.2,border=0.6,arrowsize=1.5]{->}(-3,-1.5)(-2,4)(-1,6)(3,11.5)
\rput(-3,12.5){$a$}
\rput(0,13.5){$b$}
\rput(3,12.5){$c$}
\rput(-4,-2.5){$c$}
\rput(3.5,7){$b \wlhd c$}
\rput(12,2){$a \wlhd c$}
\pscurve[linewidth=0.05]{->}(10,2.5)(6,4)(0,2.5)
\rput(9,-2.5){$(a \wlhd c) \op (b \wlhd c)$}
\end{pspicture}
\psset{unit=1mm}
\captionof{figure}{Reidemeister move~$\mathrm{VI}^u$ and induced colorings}
\label{pic:RVIu}
\end{center}
\end{proof}

\begin{corollary}\label{L:IsoQualgebraCountInvar}
Take a qualgebra $(\Q,\lhd,\op)$. An invariant of well-oriented $3$-valent knotted graphs can be constructed by assigning to such a graph the number of isosceles $\Q$-colorings $\#\ColIs_{\Q}(\D)$ of any of its diagrams~$\D$.
\end{corollary}

\begin{example}\label{EX:theta}

The Kinoshita-Terasaka $\Theta$-curve~$\Theta_{KT}$ and the standard $\Theta$-curve~$\Theta_{st}$ (Figure~\ref{pic:theta}) often serve as a litmus test for new graph invariants. One of the reasons is the following: when any edge is removed from~$\Theta_{KT}$, the remaining two ones form the unknot, just like for~$\Theta_{st}$; however, the three edges of~$\Theta_{KT}$ are knotted, in the sense that~$\Theta_{KT}$ is not isotopic to~$\Theta_{st}$. These ``partial unknottedness'' phenomena are of the same nature as those exhibited by the Borromean rings. 

Now, for these two $\Theta$-curves, consider the isosceles $\Q$-colorings of their diagrams~$\D_{KT}$ and~$D_{st}$, depicted on Figure~\ref{pic:theta}. Diagram~$\D_{st}$ (as well as all the other well-oriented versions of the underlying unoriented diagram) has $\#\Q$ isosceles $\Q$-colorings: the co-oriented arcs can be colored by any color~$x$, and the remaining arc gets the color $x \op x$. As for~$\D_{KT}$, the coloring rule~\ref{pic:SpecialCol}\rcircled{C} around $3$-valent vertices is taken into consideration in Figure~\ref{pic:theta}, and the rule~\ref{pic:Colorings}\rcircled{A}  around crossing points gives relations
$$ \left\lbrace \begin{array}{r c l}
 a &=& x \lhd  (y \op y)   =  y \lhd x, \\
 b &=& x \wlhd y  =  y \wlhd (x \op x), \\
 c &=& (y \op y) \lhd x   =  (x \op x) \wlhd y.
 \end{array}\right.$$
Thus, $\#\ColIs_{\Q}(\D_{KT})$ is the number of the solutions of the above system in~$x$ and~$y$. One easily checks that $x=y=q$ is a solution for any~$q \in \Q$ (cf.  Lemma~\ref{L:LocTrivial}). In order to find other isosceles colorings of~$\D_{KT}$, let us try the simplest case of a group qualgebra~$\Q$ and of its order $3$ elements~$x$ and~$y$. The three relations above are now equivalent to a single one, namely
$$ xyx = yxy.$$
In the symmetric group~$S_4$ for example, distinct order $3$ elements $x = (123)$ and $y= (432)$ give a solution to the above equation. One thus obtains
$$\#\ColIs_{S_4}(\D_{KT}) > \# S_4 = \#\ColIs_{S_4}(\D_{st}).$$
Since, as mentioned above, $\#\ColIs_{S_4}(\D_{st})$ is the same for all well-oriented versions of~$\D_{st}$, one concludes that~$\Theta_{KT}$ and~$\Theta_{st}$ are distinct as unoriented graphs.


\begin{center}
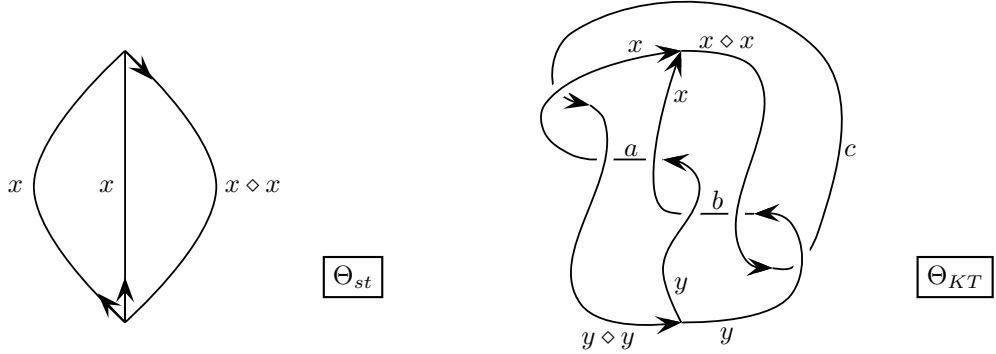

\psset{unit=1.2mm}
\begin{pspicture}(-15,-4)(45,35)
\psline[linewidth=0.2,arrowsize=2]{->}(0,0)(0,5)
\psline[linewidth=0.2](0,0)(0,30)
\pscurve[linewidth=0.2](0,0)(-10,15)(0,30)
\pscurve[linewidth=0.2](0,30)(10,15)(0,0)
\psline[linewidth=0.2,arrowsize=2]{->}(0,0)(-3,3.2)
\psline[linewidth=0.2,arrowsize=2]{->}(0,30)(3,26.8)
\rput(-12,15){$x$}
\rput(-2,15){$x$}
\rput(14,15){$x \op x$}
\rput(25,5){\psframebox{$\Theta_{st}$}}
\end{pspicture}
\begin{pspicture}(-15,-4)(35,35)
\psline[linewidth=0.2](-10,18)(-2,18)
\psline[linewidth=0.2](8,12)(0,12)
\pscurve[linewidth=0.2,arrowsize=2]{->}(10,6)(12,6)(15,10)(15,30)(-12,32)(-14,26)(-13,25)(-10,24)
\pscurve[linewidth=0.2,border=0.8,arrowsize=2]{->}(-10,18)(-12,18.5)(-15,21)(-15,24)(0,30)
\pscurve[linewidth=0.2,border=0.8,arrowsize=2]{->}(0,12)(-2,12.5)(-3,15)(0,30)
\pscurve[linewidth=0.2,border=0.8,arrowsize=2]{->}(0,30)(8,28)(7,7)(10,6)
\pscurve[linewidth=0.2,border=0.8,arrowsize=2]{->}(0,0)(12,3)(12,11)(8,12)
\pscurve[linewidth=0.2,border=0.8,arrowsize=2]{->}(0,0)(-2,6)(2,15)(1,17)(-2,18)
\pscurve[linewidth=0.2,border=0.8,arrowsize=2]{->}(-10,24)(-8.5,22.5)(-11,2)(0,0)
\rput(-5,30.5){$x$}
\rput(0,25){$x$}
\rput(5,31){$x \op x$}
\rput(-8,-2){$y \op y$}
\rput(5,-1.5){$y$}
\rput(0,4){$y$}
\rput(-5.5,19){$a$}
\rput(4,13.5){$b$}
\rput(18.5,19){$c$}
\rput(30,5){\psframebox{$\Theta_{KT}$}}
\end{pspicture}
\psset{unit=1mm}
\captionof{figure}{Isosceles colorings for diagrams of standard and Kinoshita-Terasaka $\Theta$-curves}
\label{pic:theta}
\end{center}

\end{example}

\subsection*{A variation of qualgebra ideas}

Restricting our attention to isosceles colorings only, we do not exploit the whole structure of qualgebra. Indeed, the only values of $a \op b$ we need are those for $a = b$. In other words, we use only the ``squaring'' part $\sq:a \mapsto a \op a$ of the operation~$\op$. Pursuing this remark, let us try to determine for which unary operations~$\sq$ the coloring rule~\ref{pic:Colorings}\rcircled{C} is topological. 

One arrives to the following notion:

\begin{definition}\label{D:Squandle}
A set~$\Q$ endowed with a binary operation~$\lhd$ and a unary operation~$\sq$ (which we often denote by $a \mapsto a^2$) is called a \emph{squandle} if it satisfies Axioms \eqref{E:SD}-\eqref{E:Idem} and \eqref{E:SQA1}-\eqref{E:SQA2} (see page~\pageref{E:SD}).
\end{definition}

The term ``squandle'' (similarly to the term ``qualgebra'') comes from terms ``square'' and ``quandle''  zipped together, cf. Figure~\ref{pic:Term}.

Let us also note the compatibility relations between operations~$\sq$ and~$\wlhd$:

\begin{lemma}\label{L:SQualgebraProperties}
A squandle $(\Q,\lhd,\sq)$ enjoys the following properties:
\begin{align}
&a\wlhd b^2 = (a\wlhd b)\wlhd b,\label{E:SQA1'}\\
&a^2 \wlhd b = (a\wlhd b)^2.\label{E:SQA2'}
\end{align}
\end{lemma}

\begin{example}\label{EX:QualgebraAsSquandle}
A qualgebra $(\Q,\lhd,\op)$ always gives rise to a squandle $(\Q,\lhd,\sq:a \mapsto a \op a)$. Moreover, the sub-squandles of the latter (which are not necessarily sub-qualgebras) can be of interest. In particular, conjugation and squaring operation $a \mapsto a^2$ in a group form a squandle, called a \emph{group squandle}. Axioms \eqref{E:SQA1}-\eqref{E:SQA2} can now be seen as an abstraction of the relations between conjugation and squaring operations in a group.
\end{example}

Now, considering squandle colorings, one gets the following results, with the statements and proofs analogous to the qualgebra case:

\begin{proposition}\label{P:SquandleColor}
Take a set~$\Q$ endowed with a binary operation~$\lhd$ and a unary operation~$\sq$. Coloring rules from Figure~\ref{pic:Colorings}\rcircled{A}\&\rcircled{C} are topological if and only if $(\Q,\lhd,\sq)$ is a squandle.
\end{proposition}

\begin{corollary}\label{L:SquandleCountInvar}
Take a squandle $(\Q,\lhd,\sq)$ and consider $\Q$-coloring rules~\ref{pic:Colorings}\rcircled{A}\&\rcircled{C}. The (possibly infinite) quantity $\#\Col_{\Q}(\D)$ does not depend on the choice of a diagram~$\D$ representing a well-oriented $3$-valent knotted graph~$\Gamma$.
\end{corollary}

\begin{example}\label{EX:thetaSquandle}
Let us resume Example~\ref{EX:theta}. In the symmetric group~$S_4$, consider the subset~$S^3_4$ of cycles of length~$3$. It contains $8$ elements, and it is closed under conjugation and squaring. Hence~$S^3_4$, endowed with conjugation and squaring operations, is a size~$8$ squandle (but not a qualgebra, since it does not contain $\Id = (123)^3$). Calculations from Example~\ref{EX:theta} show that $\#\Col_{S^3_4}(\D_{st}) = \# S^3_4 = 8,$ and that $\#\Col_{S^3_4}(\D_{KT})$ is the number of solutions of $xyx=yxy$ in~$S^3_4$. Now, for any~$x$, the pair $(x,x)$ is a solution, while $(x,x^{-1})$ is not. Further, we have seen that cycles $(123)$ and $(432)$ form a solution, and one checks that  $(123)$ and $(423)$ do not. A conjugation argument allows to conclude that for a fixed~$x_0$, precisely a half of the pairs $(x_0,y)$ are solutions, which totals to $\#\Col_{S^3_4}(\D_{KT}) = 8\cdot 4 = 32$. Thus, although this example gives nothing new about the graphs~$\Theta_{KT}$ and~$\Theta_{st}$ (the group qualgebra of~$S_4$ was sufficient to distinguish them), it does show that with squandle colorings, actual computation of counting invariants can be much easier.
\end{example}

\section{Qualgebras and squandles with $4$ elements}\label{sec:Qualgebras4}

In this section we completely describe qualgebras and squandles with $4$ elements. Compared to quandles, these new structures come with abundant examples even in such a small size.

\subsection*{General properties}

Some general facts about qualgebras and squandles are necessary before proceeding to classification questions.

\begin{notation}\label{D:Translations}
Given a quandle $(\Q,\lhd)$ (in particular, a qualgebra or squandle) and an $a \in \Q$, denote by~$\S_a$ the \textit{right translation} map $x \mapsto x \lhd a$. We write quandle maps on the right of their arguments, e.g., $(x)\S_a =  x \lhd a$.
\end{notation} 

Most axioms of quandle-like structures can be expressed in terms of these right translations, allowing one to work with symmetric groups instead of abstract structures. This approach was extensively used for quandles in \cite{LopesRoseman}. Here we apply similar ideas to qualgebras and squandles.

\begin{lemma}\label{L:Sa}
Given a qualgebra $(\Q,\lhd,\op)$ or a squandle $(\Q,\lhd,\sq)$, the map
\begin{align}
\S: \Q &\longrightarrow \Aut(\Q),\label{E:MapS}\\
a &\longmapsto \S_a \notag
\end{align}
is a well-defined qualgebra/squandle morphism from~$\Q$ to $\Aut(\Q)$, the latter being the group qualgebra/squandle of qualgebra/squandle automorphisms of~$\Q$. 
\end{lemma}

\begin{proof}
We prove the qualgebra version of the assertion, the squandle one being analogous.

One should first show that any~$\S_a$ is a qualgebra automorphism. Indeed, it is invertible due to Axiom~\eqref{E:Inv}, its inverse~$\S_a^{-1}$ being the map $x \mapsto x \wlhd a$, and it respects operations~$\lhd$ and~$\op$ due to~\eqref{E:SD} and~\eqref{E:QA2} respectively.

It remains to prove that~$\S$ is a qualgebra morphism. Relation $\S_{a \op b} = \S(a) \S(b)$ directly follows from~\eqref{E:QA1}. Next, for any $x \in \Q$ one calculates (using quandle Axioms \eqref{E:SD}-\eqref{E:Idem})
$$(x)\S_{a \lhd b} = x \lhd (a \lhd b) = ((x\wlhd b) \lhd b) \lhd (a \lhd b) = ((x\wlhd b) \lhd a) \lhd b = (((x) \S_b^{-1})\S_a)\S_b = (x) (\S_a \lhd \S_b),$$
since in the group qualgebra $\Aut(\Q)$ operation~$\lhd$ is the conjugation. Hence, $\S_{a \lhd b} = \S(a) \lhd \S(b)$.  
\end{proof}

\begin{lemma}\label{L:Sa'}
For a finite qualgebra $\Q$, the image $\S(\Q)$ of the map~\eqref{E:MapS} is a subgroup of $\Aut(\Q)$.
\end{lemma}

\begin{proof}
Since~$\S$ is a qualgebra morphism (Lemma~\ref{L:Sa}), its image $\S(\Q)$ is a sub-qualgebra of the group qualgebra $\Aut(\Q)$, which is finite since~$\Q$ is finite. Let us now show that, in general, a non-empty finite sub-qualgebra~$R$ of a group qualgebra~$G$ is in fact a subgroup. Indeed, $R$ is stable under product since it is a sub-qualgebra; it contains the unit~$1$ of the group~$G$ since $1 = a^p$, where~$a$ is any element of~$R$ and~$p$ is its order in~$G$; and it contains all the inverses, since, with the previous notation, $a^{-1} = a^{p-1}$.
\end{proof}

Note that this lemma is false for squandles in general: a counter-example will be given below.

In a study of a qualgebra or squandle, the understanding of its local structure can be useful.

\begin{notation}\label{D:Local}
Take a qualgebra or a squandle~$\Q$ and an $a \in \Q$.
\begin{itemize}
\item The sub-qualgebra/sub-squandle of~$\Q$ generated by~$a$ is denoted by~$\Q_a$.
\item The set of fixed points~$x$ of~$\S_a$ (i.e., $(x)\S_a = x$) is denoted by $\FP(a)$.
\item The set of elements~$x$ of~$\Q$ fixing~$a$ (in the sense that $(a)\S_x = a$) is denoted by $\St(a)$.
\end{itemize}
\end{notation}

\begin{lemma}\label{L:FPandStab}
Take a qualgebra $(\Q,\lhd,\op)$ or a squandle $(\Q,\lhd,\sq)$, and an $a \in \Q$. The sets $\FP(a)\subseteq \Q$ and $\St(a)\subseteq \Q$ are both sub-qualgebras/sub-squandles of~$\Q$ containing~$\Q_a$.
\end{lemma}

\begin{proof}
The assertion about $\FP(a)$ being a sub-qualgebra/sub-squandle of~$\Q$ holds true because~$\S_a$ is a qualgebra/squandle automorphism of~$\Q$. As for $\St(a)$, note that the set $\widetilde{\St}(a)$ of maps in $\Aut(\Q)$ stabilizing~$a$ is a subgroup of $\Aut(\Q)$, hence also a sub-qualgebra/sub-squandle, so $\St(a)$, which is its pre-image $\S^{-1}(\widetilde{\St}(a))$ along the qualgebra/squandle morphism~$\S$, is a sub-qualgebra/sub-squandle of~$\Q$ (cf. Lemma~\ref{L:Sa}).

Further, both $\FP(a)$ and $\St(a)$ contain~$a$ due to the idempotence axiom~\eqref{E:Idem}. Since they were both shown to be sub-qualgebras/sub-squandles of~$\Q$, they have to include the whole~$\Q_a$.
\end{proof}

\begin{lemma}\label{L:TrivialQuandle}
Consider a set~$\Q$ endowed with a {trivial} quandle operation $a \lhd_0 b = a$. Then any unary operation~$\sq$ completes it into a squandle. Further, a binary operation~$\op$ completes it into a qualgebra if and only if~$\op$ is commutative. 
\end{lemma}

\begin{proof}
With the trivial quandle operation, all qualgebra and squandle axioms automatically hold true except for~\eqref{E:QAComm}, which is equivalent to the commutativity of~$\op$. 
\end{proof}

\begin{definition}\label{D:TrivialQuandle}
The qualgebras/squandles from the lemma above are called \emph{trivial}.
\end{definition}

Observe that colorings by trivial qualgebras/squandles do not distinguish over-crossings from under-crossings, hence the corresponding counting invariants can capture only the underlying abstract graph and not the way it is knotted in~$\RR^3$. However, weight invariants can be sensible to the knotting information even for trivial structures.

In size $3$, all qualgebras/squandles turn out to be trivial:

\begin{proposition}\label{L:Qualgebra3}
A non-trivial qualgebra or squandle has at least $4$ elements.
\end{proposition}

\begin{proof}
Let~$\Q$ be a non-trivial qualgebra or squandle, and $a$ be its element with non-trivial right translation~$\S_a$. Then $\S_{a^2} = \S_a^2$ is different from~$\S_a$, so $\FP(a)$ contains at least $2$ distinct elements~$a$ and~$a^2$ (cf. Lemma~\ref{L:FPandStab}). Further, since $S_a \in \Aut(\Q)$ is not the identity, at least two elements of~$\Q$ should lie outside $\FP(a)$. Altogether, one gets at least $4$ elements.
\end{proof}

We finish by showing that every qualgebra/squandle is ``locally trivial'':

\begin{lemma}\label{L:LocTrivial}
Take a qualgebra $(\Q,\lhd,\op)$ or a squandle $(\Q,\lhd,\sq)$, and an $a \in \Q$. 
The sub-qualgebra/sub-squandle~$\Q_a$ of~$\Q$ generated by~$a$ is trivial. In the qualgebra case, the restriction of operation~$\op$ to~$\Q_a$ is commutative. 
\end{lemma}

\begin{proof}
Lemma~\ref{L:FPandStab} shows that every $x \in \Q_a$ fixes~$a$. Thus, the set $\FP(x)$ contains~$a$; but, being a sub-qualgebra/sub-squandle of~$\Q$ (again due to Lemma~\ref{L:FPandStab}), it should contain the whole~$\Q_a$. The triviality of~$\lhd$ restricted to~$\Q_a$ follows. The commutativity of~$\op$ on~$\Q_a$ is now a consequence of Lemma~\ref{L:TrivialQuandle}.
\end{proof}

\subsection*{Classification of qualgebras of size $4$}

Since trivial qualgebras/squandles were completely described in Lemma~\ref{L:TrivialQuandle}, only non-trivial structures are studied in the remainder of this section.

We start with a full list of $9$ non-trivial qualgebra structures on a $4$ element set $\P = \{p,q,r,s\}$ (up to isomorphism). Involution
\begin{equation}\label{E:tau}
(p)\rrho = q, (q)\rrho = p, (r)\rrho = r, (s)\rrho = s
\end{equation}
will be used in this description.

\begin{proposition}\label{P:Qualgebras4}
Any non-trivial qualgebra with $4$ elements is isomorphic to the set~$\P$ with the following operations (here~$x$ and~$y$ are arbitrary elements of~$\P$):
\begin{align*}
&x \lhd r = (x)\rrho, & x \lhd y &= x \text{  if } y \neq r;&\\
& r \op r = s, & r \op x &= x \op r = r \text{  if } x \neq r, & \\
& s \op s = s, & q \op s &= s \op q \in \{p,q,s\}, & p \op s &= s \op p = (q \op s)\rrho,\\
& p \op q = q \op p = s, & q \op q &\in \{p,q,s\}, & p \op p &= (q \op q)\rrho.
\end{align*}
Moreover, for any choices of $q \op s$ and $q \op q$ in $\{p,q,s\}$, the resulting structure is a qualgebra.
\end{proposition}

In order to better feel the qualgebra structures from the proposition, think of the element~$r$ as the \underline{r}otation (of~$p$ or~$q$), and of~$s$ as the \underline{s}quare (of~$r$).

\begin{proof}
Fix a qualgebra structure on~$\P$. Observe first that for any $x \in \P$, one has $\#\FP(x)\ge 2$. Indeed, otherwise the sub-qualgebra~$\P_x$ generated by~$x$, which is contained in $\FP(x)$ due to Lemma~\ref{L:FPandStab}, would consist of~$x$ itself only, and so, according to Lemma~\ref{L:Sa}, $\S(\{\P_x\}) =\{\S_x\}$ would be a $1$-element sub-qualgebra of $\Aut(\P) \subseteq S_4$, which is possible only if $\S_x = \Id$, giving $\#\FP(x) = 4$. 

Now, condition $\#\FP(x)\ge 2$ implies that $\S_x$ moves at most $2$ elements of~$\P$, so it is a transposition or the identity. But then $\S(\P)$ is a subgroup of~$S_4$ (Lemma~\ref{L:Sa'}) containing nothing except transpositions and the identity, hence either $\S(\P) = \{\Id\}$ (and thus the the qualgebra is trivial), or, without loss of generality,   
$$\S(\P) = \{\Id,\rrho\},$$ 
with, say, $\S_r = \rrho$. We next show that $\S^{-1}(\rrho)$ consists of~$r$ only. Indeed, 
$\S(\P_r)$ is a sub-qualgebra of $\Aut(\P)$ (Lemma~\ref{L:Sa}) contained in $\S(\FP(r))$ (Lemma~\ref{L:FPandStab}), so $\S(\FP(r)) = \{\S(r),\S(s)\} = \{\rrho,\S_s\}$ should include $\rrho^2=\Id$, hence $\S_s = \Id$, implying $s \notin \S^{-1}(\rrho)$. As for~$p$ and~$q$, they are not fixed by~$\rrho$, so they cannot lie in $\S^{-1}(\rrho)$.

We can thus restrict our analysis to the case $\S_r = \rrho$ and $\S_y = \Id$ for $y \neq r$. This choice of operation~$\lhd$ guarantees~\eqref{E:Inv} and~\eqref{E:Idem}. Axiom~\eqref{E:SD} can be checked directly, but we prefer recalling that it is a consequence of \eqref{E:QA1}-\eqref{E:QAComm}.

Let us now analyze specific qualgebra axioms \eqref{E:QA1}-\eqref{E:QAComm}. First, \eqref{E:QA1} translates as $\S_{b \op c} = \S_b \S_c$, which here means that $r \op x = x \op r = r$ for all $x \neq r$, while all other products take value in $\{p,q,s\}$. Next, \eqref{E:QA2} is equivalent to all maps from $\S(\P)$ respecting the operation~$\op$, which here translates as $(a \op b)\rrho = (a)\rrho \op (b)\rrho$. This means that $r \op r$ and $s \op s$ are both $\rrho$-stable, so, lying in $\{p,q,s\}$, they can equal only~$s$; this gives nothing new when one of $a,b$ is~$r$ and the other one is not; and it divides the remaining ordered couples into pairs, with the product for one couple from the pair determined by that for the other (e.g., $p \op s = (q \op s)\rrho$). At last, \eqref{E:QAComm} is automatic when one of the elements~$a$ and~$b$ is~$r$ and the other one is~$p$ or~$q$, and for the other couples it means the commutativity of~$\op$. In particular, this commutativity gives $p \op q = q \op p$, which, combined with $(p \op q)\rrho = (p)\rrho \op (q)\rrho = q \op p$, implies that $p \op q$ is $\rrho$-stable, so, lying in $\{p,q,s\}$, it can equal only~$s$. Putting all these conditions together, one gets the description of~$\op$ given in the statement.

It remains to check that the $9$ qualgebra structures obtained are pairwise non-isomorphic. Let $f:\P \rightarrow \P$ be a bijection intertwining structures $(\lhd,\op_1)$ and $(\lhd,\op_2)$ from our list. Since~$r$ is the only element of~$\P$ with $\S_a \neq \Id$, one has $(r)f=r$, and also $(s)f=(r \op_1 r)f = r \op_2 r =s$. Two options emerge: either $(q)f=q$ and $(p)f=p$, in which case $\op_1$ and $\op_2$ automatically coincide; or $(q)f=p$ and $(p)f=q$, that is, $f = \rrho$, in which case one has
$$x \op_2 y = ((x)f^{-1} \op_1 (y)f^{-1} )f = ((x)\rrho^{-1} \op_1 (y)\rrho^{-1} )\rrho = x \op_1 y,$$ 
since, being a right translation, $\rrho = \S_r$ respects~$\op_1$. One concludes that there are no isomorphisms between different qualgebra structures from our list.
\end{proof}

\subsection*{Properties and examples}

In spite of very close definitions, the $9$ structures above exhibit quite different algebraic properties. Some of them are studied below.

\begin{proposition}\label{P:Qualgebras4Prop}
The operations~$\op$ from Proposition~\ref{P:Qualgebras4} are
\begin{itemize}
\item all commutative;
\item never cancellative;
\item unital if and only if $q \op s = s \op q = q$ and $p \op s = s \op p = p$;
\item associative if and only if $q \op s = s \op q = p \op s = s \op p = s$ and either $q \op q = p \op p = s$, or $q \op q = q$ and $p \op p = p$;
\item never unital associative.
\end{itemize}
\end{proposition}

\begin{proof}
The commutativity is read from the explicit definition of~$\op$. The non-cancellativity follows from the ``absorbing'' property of the element~$r$ with respect to~$\op$.

Further, relations $q \op p = s$ and $r \op s = r$ imply that~$s$ is the only possible neutral element. Examining the definition of~$\op$, one sees that it is indeed so if and only if the value of $q \op s = s \op q$ is chosen to be~$q$ (implying $p \op s = s \op p = (q \op s)\rrho = (q)\rrho = p$).

Associativity is trickier to deal with. First, if~$\op$ is associative, then $s \op q$ has to equal~$s$:
$$s \op q = (r \op r) \op q = r \op (r \op q) = r \op r = s.$$
Since $(s)\rrho =s$, this implies $q \op s = p \op s = s \op p = s$. Next, $q \op q$ can not be~$p$, since this would give 
$$ q = (p)\rrho = (q \op q)\rrho = (q)\rrho \op (q)\rrho = p \op p = p \op (q \op q) = (p \op q) \op q = s \op q = s.$$
Thus, either $q \op q = p \op p = s$, or $q \op q = q$ and $p \op p = p$. It remains to show that these two operations~$\op$ are indeed associative. Consider the direct product $\ZZ_4^{\times 3}$ endowed with the term-by-term multiplication~$\cdot$, and define an injection $\P \hookrightarrow \ZZ_4^{\times 3}$ by
\begin{align*}
p &\mapsto (a,0,1), & r &\mapsto (0,0,3),\\
q &\mapsto (0,a,1), & s &\mapsto (0,0,1),
\end{align*}
for some $a \neq 0$. One easily checks that this injection intertwines operations~$\op$ and~$\cdot$, where one takes $a = 2$ for the choice $q \op q = p \op p = s$, and $a = 1$ for the choice $q \op q = q$, $p \op p = p$. Thus the associativity of~$\cdot$ implies that of~$\op$.  

To conclude, notice that if a unital associative~$\op$ existed, then it would satisfy incompatible conditions $q \op s = q$ and $q \op s = s$.
\end{proof}

Thus, $3$ non-trivial qualgebra structures with $4$ elements are unital, and $2$ are associative. Further, non of these qualgebras can be a sub-qualgebra of a group qualgebra because of the non-cancellativity.

\begin{example}\label{EX:cuffs}
Let us now use the $4$-element qualgebras obtained above for distinguishing the standard cuff graph~$\Cuff_{st}$ from the Hopf cuff graph~$\Cuff_{H}$. Consider their diagrams~$\D_{st}$ and~$\D_H$ depicted on Figure~\ref{pic:cuffs}, and choose the qualgebra~$\P$ from Proposition~\ref{P:Qualgebras4} with $q \op q = s$ and $q \op s = q$. The multiplication~$\op$ of this qualgebra can be briefly described by saying that it is commutative with a neutral element $s$, that the element~$r$ absorbs everything but itself (in the sense that $r \op x = r$), and that $x \op y = s$ for $x=y$ and for $x = (y)\rrho$.

With the orientation on Figure~\ref{pic:cuffs}, the coloring rules for~$\D_{st}$ around $3$-valent vertices read $b \op a = a$ and $b \op c = c$. Further, note that every orientation of~$\D_{st}$ is a well-orientation, and that an orientation change results only in an argument inversion in one or all of the relations above; since~$\op$ is commutative, this does not change the relations. Summarizing, for any orientation of~$\D_{st}$ one gets a bijection
$$\Col_{\P}(\D_{st}) \overset{bij}{\longleftrightarrow} \{(a,b,c) \in \P \,|\, b \op a = a, \,b \op c = c\}.$$
Now, equation $b \op a = a$ (and similarly $b \op c = c$) has $6$ solutions in~$\P$: either~$b$ is the unit~$s$, and~$a$ is arbitrary; or~$b$ is~$p$ or~$q$, and $a=r$. Searching for pairs of solutions with the same~$b$, one gets
$$\Col_{\Q}(\D_{st}) \overset{bij}{\longleftrightarrow} \{(a,s,c) \,|\, a,c \in \Q \} \bigsqcup \{(r,b,r) \,|\, b \in \{p,q\} \},$$
and so $\#\Col_{\P}(\D_{st}) = 4 \cdot 4 + 2 = 18$.

Let us now turn to the Hopf cuff graph diagram~$\D_H$, oriented as shown on Figure~\ref{pic:cuffs}. Coloring rules around crossing points allow one to express~$a'$ and~$c'$ in terms of other colors: $c' = c \lhd a$, $a' = a \wlhd c'$. In our qualgebra~$\P$, all the translations~$\S_x$ (recall Notation~\ref{D:Translations}) are either the identity or~$\rrho$, so they are pairwise commuting involutions, implying $a' = a \wlhd c'= (a)\S_{c \lhd a}^{-1} = (a)\S_{c \lhd a} = (a)(\S_c \lhd \S_a) = (a)\S_c =  a \lhd c$. Further, around $3$-valent vertices coloring rules give $b \op a = a'$ and $b \op c = c'$. Using the preceding remarks, this gives
$$\Col_{\P}(\D_{H}) \overset{bij}{\longleftrightarrow} \{(a,b,c) \in \P \,|\, b \op a = a \lhd c, \,b \op c = c \lhd a\}.$$
The latter system admits no solutions with $b=r$. For $b=s$, the equations become $ a = a \lhd c$ and $c = c \lhd a$, for which the solutions are all pairs $(a,c)$ except $a=r$, $c \in \{p,q\}$ or vice versa. In the remaining case $b \in \{p,q\}$, the only possibility is $a = c =r$.
Summarizing, one gets
$$\Col_{\P}(\D_{H}) \overset{bij}{\longleftrightarrow} \{(a,s,c) \,|\, a,c \in \{p,q,s\} \}\bigsqcup \{(r,s,r),(r,s,s),(s,s,r) \} \bigsqcup \{(r,b,r) \,|\, b \in \{p,q\} \},$$
and so $\#\Col_{\P}(\D_{H}) = 3 \cdot 3 + 3 + 2 = 14 \neq \#\Col_{\P}(\D_{st})$. With the orientation remarks made for~$\D_{st}$, Corollary~\ref{L:QualgebraCountInvar} now guarantees that the two unoriented cuff graphs are not mutually isotopic.

\begin{center}
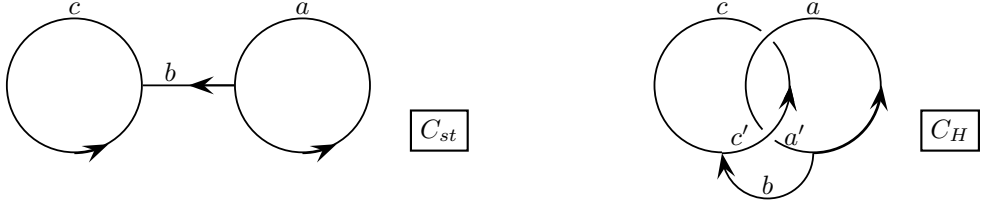
 
\psset{unit=1.2mm}
\begin{pspicture}(-10,-13)(60,11)
\psline[linewidth=0.2](17.5,0)(7.5,0)
\psline[linewidth=0.2,arrowsize=2]{->}(17.5,0)(12.5,0)
\pscircle[linewidth=0.2](0,0){7.5}
\psarc[linewidth=0.2,arrowsize=2]{->}(0,0){7.5}{270}{300}
\pscircle[linewidth=0.2](25,0){7.5}
\psarc[linewidth=0.2,arrowsize=2]{->}(25,0){7.5}{270}{300}
\rput(25,8.5){$a$}
\rput(0,8.5){$c$}
\rput(10.5,1.5){$b$}
\rput(40,-5){\psframebox{$\Cuff_{st}$}}
\end{pspicture}
\begin{pspicture}(-10,-13)(30,11)
\pscircle[linewidth=0.2](0,0){7.5}
\pscircle[linewidth=0.2,border=0.8](10,0){7.5}
\psarc[linewidth=0.2,arrowsize=2,border=0.8]{->}(0,0){7.5}{-90}{0}
\psarc[linewidth=0.2,arrowsize=2]{->}(10,0){7.5}{-90}{0}
\psarc[linewidth=0.2,arrowsize=2]{<-}(5,-7.5){5}{180}{0}
\rput(10,8.5){$a$}
\rput(0,8.5){$c$}
\rput(5,-11){$b$}
\rput(8,-5.5){$a'$}
\rput(2,-5.5){$c'$}
\rput(25,-5){\psframebox{$\Cuff_{H}$}}
\end{pspicture}
\psset{unit=1mm}
\captionof{figure}{Qualgebra colorings for the diagrams of standard and Hopf cuff graphs} \label{pic:cuffs}
\end{center}
\end{example}

\subsection*{Classification of squandles of size $4$}

Let us now turn to non-trivial $4$-element squandle structures. We shall see that $3$ out of the $4$ of them are induced from the qualgebra structures from Proposition~\ref{P:Qualgebras4} according to the procedure described in Example~\ref{EX:QualgebraAsSquandle}.

\begin{proposition}\label{P:Squandles4}
Any non-trivial squandle with $4$ elements is isomorphic
\begin{itemize}
\item either to the sub-squandle~$S^2_3$ of the group squandle of the symmetric group~$S_3$ consisting of the identity and the transpositions $(12)$, $(23)$ and $(13)$;
\item or to the set $\P=\{p,q,r,s\}$ with the following operations (here~$x$ and~$y$ are arbitrary elements of~$\P$, and~$\rrho$ is the involution defined by~\eqref{E:tau}):
\begin{align*}
&x \lhd r = (x)\rrho, & &x \lhd y = x \text{  if } y \neq r;&\\
&r^2 = s^2 = s, &  &q^2 \in \{p,q,s\}, & p^2 = (q^2)\rrho.
\end{align*}
\end{itemize}
\end{proposition}

\begin{proof}
Fix a squandle structure on~$\P$. 
Repeating verbatim the beginning of the proof of Proposition~\ref{P:Qualgebras4}, one shows that, for any $x \in \P$, $\S_x$ is a transposition or the identity. Forgetting trivial squandles, which correspond to $\S(\P) = \{\Id\}$, consider three remaining cases.
\begin{enumerate}
\item There are two intersecting transpositions --- say,  $(p,q)$ and $(q,r)$ --- in $\S(\P)$. Then $\S(\P)$ also contains $(p,q) \lhd (q,r) = (q,r)(p,q)(q,r) = (p,r)$ and $(p,q)^2 = \Id$. Since~$\P$ itself has only $4$ elements, this implies that~$\S$ is an injection, so, as a squandle, $\P$ is isomorphic to the sub-squandle of~$S_4$ formed by $\Id$, $(12)$, $(23)$ and $(13)$ (which is indeed a sub-squandle since stable by conjugation and squaring). Omitting the element $4$, one sees that the latter sub-squandle of~$S_4$ is isomorphic to the sub-squandle~$S^2_3$ of~$S_3$.
\item There are two non-intersecting transpositions --- say,  $(p,q)$ and $(r,s)$ --- in $\S(\P)$. A fixed point argument shows that  $(p,q) \in \{\S_r, \S_s\}$ and $(r,s)\in \{\S_p, \S_q\}$. Suppose for instance that $\S_r = (p,q)$ and $\S_p = (r,s)$. Consider now the possible values of~$r^2$. According to Lemma~\ref{L:FPandStab}, one has $r^2 \in \FP(r) = \{r,s\}$. On the other hand, $\S_{r^2} = (\S_r)^2 = \Id$, thus $r^2 \neq r$, leaving only the possibility $r^2 = s$. Thus, $\S_s = (\S_r)^2 = \Id$. But this leads to a contradiction with~\eqref{E:SD}:
$(q \lhd r) \lhd p = p \lhd p = p$, but $(q \lhd p) \lhd (r \lhd p) = q \lhd s = q$. Hence this case  does not lead to squandle structures.
\item The only remaining situation is $\S(\P) = \{\Id,\rrho\}$ with, say, $\S_r = \rrho$. Repeating once again an argument from the proof of Proposition~\ref{P:Qualgebras4}, one concludes that operation~$\lhd$ is defined by $\S_r = \rrho$ and $\S_x = \Id$ for $x \neq r$. In Proposition~\ref{P:Qualgebras4}, this operation was shown to satisfy \eqref{E:SD}-\eqref{E:Idem}. Thus
only specific squandle axioms \eqref{E:SQA1}-\eqref{E:SQA2} remain to be checked. First, \eqref{E:SQA1} translates as $\S_{b^2} = \S_b^2$, which here means that $x^2 \in \{p,q,s\}$ for all $x \in \P$. Next, \eqref{E:SQA2} is equivalent to all maps from $\S(\P)$ respecting the operation~$\sq$, which here translates as $(a^2)\rrho = ((a)\rrho)^2$. This means that $p^2 = (q^2)\rrho$, and that $r^2$ and $s^2$ are both $\rrho$-stable, so, lying in $\{p,q,s\}$, they can equal only~$s$. One thus gets the description of~$\sq$ given in the statement.
\end{enumerate}
The four structures obtained are shown to be mutually non-isomorphic in the same way as it was done for qualgebras in Proposition~\ref{P:Qualgebras4}.
\end{proof}

Note that the first structure from the proposition is an example of a sub-squandle of a group squandle (here of~$S_3$) which is not a subgroup, showing that Lemma~\ref{L:Sa'} does not hold for squandles.

\section{Qualgebra $2$-cocycles and weight invariants of graphs}\label{sec:QualgebraHom}

We now return to the general settings of a qualgebra $(\Q,\lhd,\op)$ and $\Q$-colorings of well-oriented knotted $3$-valent graph diagrams, according to coloring rules from Figure~\ref{pic:Colorings}\rcircled{A}\&\rcircled{B}. The aim of this section is to extract weight invariants out of such colorings.

\subsection*{Qualgebra $2$-cocycles as Boltzmann weight functions}

Recall the type of weight functions used for quandle colorings of knot diagrams (Example~\ref{EX:QuandleCol2}): starting with a map~$\w:\Q\times\Q\rightarrow\ZZ$, we applied it to the colors of two arcs around crossing points, the arcs being chosen according to Figure~\ref{pic:Quandles}. Note that the colors of these two arcs determine all other colors around a crossing point. Trying to treat $3$-valent vertices in a similar way, take a map~$\wl:\Q\times\Q\rightarrow\ZZ$, and let $(\w,\wl)$ be a weight function defined on crossing points as above, and on $3$-valent vertices according to Figure~\ref{pic:BoltzmannGraph}. Note that, like for crossing points, we take into consideration the colors of the arcs which determine all other colors around a $3$-valent vertex. Remark also that unrelated maps~$\wl$ and~$\wlinv$ could be chosen for unzip and zip vertices; our choice simplifies further calculations, however conserving abundant examples. To make our notations easier to follow, for denoting the components of our weight function we chose Greek letters with a shape referring to that of corresponding special points.

\begin{center}
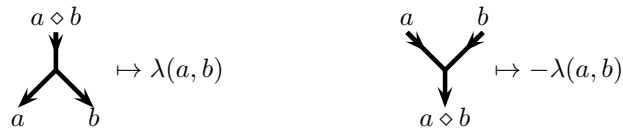

\begin{pspicture}(50,13)
\psline[linewidth=0.6,arrowsize=2]{->}(5,10)(5,7)
\psline[linewidth=0.6,arrowsize=2]{->}(5,10)(5,5)(0,0)
\psline[linewidth=0.6,arrowsize=2]{->}(5,5)(10,0)
\rput[b](5,11){$a \op b$}
\rput[t](0,0){$\bw a \bw$}
\rput[t](10,0){$b$}
\rput(20,5){$\mapsto \wl(a,b)$}
\end{pspicture}
\begin{pspicture}(15,13)
\psline[linewidth=0.6,arrowsize=2]{->}(0,10)(2.5,7.5)
\psline[linewidth=0.6,arrowsize=2]{->}(0,10)(5,5)(5,0)
\psline[linewidth=0.6,arrowsize=2]{->}(10,10)(7.5,7.5)
\psline[linewidth=0.6](10,10)(5,5)
\rput[b](0,11){$a$}
\rput[b](10,11){$b$}
\rput[t](5,0){$a \op b$}
\rput(20,5){$\mapsto -\wl(a,b)$}
\end{pspicture}
\captionof{figure}{Weight function for qualgebra-colored graph diagrams}\label{pic:BoltzmannGraph}
\end{center}

\begin{proposition}\label{P:BoltzmannGraph}
Take a qualgebra $(\Q,\lhd,\op)$ and two maps $\w,\wl:\Q\times\Q\rightarrow\ZZ$. The weight function $(\w,\wl)$ described above (and depicted on Figures~\ref{pic:Quandles} and~\ref{pic:BoltzmannGraph}) is Boltzmann if and only if it satisfies, for all elements of~$\Q$, Axioms \eqref{E:BWR3}-\eqref{E:BWR1} together with three additional ones: 
\begin{align}
&\w(a,b \op c) = \w(a,b) + \w(a\lhd b,c), \label{E:BWRIV}\\
&\w(a\op b, c) + \wl(a \lhd c, b \lhd c) = \w(a,c) +\w(b,c) + \wl(a,b), \label{E:BWRVI}\\ 
&\w(a,b) + \wl(a,b) = \wl(b,a \lhd b). \label{E:BWRV}
\end{align}
\end{proposition}

\begin{proof}
One should check when each of the six R-moves, combined with the induced coloring transformation from Definition~\ref{D:TopColRules}, leaves the $(\w,\wl)$-weights unchanged. For moves R$\mathrm{I}$-R$\mathrm{III}$, this is known to be equivalent to Axioms \eqref{E:BWR3}-\eqref{E:BWR1} for~$\w$ (cf. Example~\ref{EX:QuandleCol2}). Figure~\ref{pic:AxiomsQu2Cocycles} deals with the zip versions of moves R$\mathrm{IV}$-R$\mathrm{VI}$, the unzip versions being similar due to our choice of weight function around zip and unzip vertices, and to Relations \eqref{E:QA1'}-\eqref{E:QAComm'} allowing to treat operation~$\wlhd$ in a manner analogous to~$\lhd$. In the figure, move R$\mathrm{IV}^z$ (respectively, R$\mathrm{VI}^z$ or R$\mathrm{V}^z$) is shown to preserve weights if and only if~\eqref{E:BWRIV} (respectively, \eqref{E:BWRVI} or~\eqref{E:BWRV}) is satisfied.

\begin{center}
\psset{unit=2mm}
\begin{pspicture}(-12.5,-3.5)(18,13.5)
\pscurve[linewidth=0.2,arrowsize=1.2]{->}(-3,11.5)(-1.5,1.5)(3,-1.5)
\psline[linewidth=0.2,arrowsize=1.2]{->}(3,11.5)(2.5,9)
\pscurve[linewidth=0.2](2.5,9)(2.3,8)(1.5,6)(0.5,5)
\psline[linewidth=0.2,arrowsize=1.2]{->}(0,12.5)(-0.5,10)
\pscurve[linewidth=0.2](-0.5,10)(-0.6,9.5)(-0.4,7)(0.5,5)
\psline[linewidth=0.2,border=0.6,arrowsize=1.2]{->}(0.5,5)(-3,-1.5)
\rput(-3,12.5){$a$}
\rput(0,13.5){$b$}
\rput(3,12.5){$c$}
\rput(-4,-2.5){$b \op c$}
\rput(1.5,3){$b \op c$}
\rput(5.3,-2.7){$a \lhd (b \op c)$}
\psline[linewidth=0.05]{|->}(0,5)(-6,5)
\rput(-11,5){\psframebox{$-\wl(b, c)$}}
\psline[linewidth=0.05]{|->}(-2.5,1.5)(-6,1.5)
\rput(-11,1.5){\psframebox{$\w(a,b \op c)$}}
\rput(14,5){$\overset{\text{R}\mathrm{IV}^z}{\longleftrightarrow}$}
\end{pspicture}
\begin{pspicture}(-7.5,-3.5)(25,13.5)
\pscurve[linewidth=0.2,arrowsize=1.2]{->}(-3,11.5)(1,6)(2,4)(3,-1.5)
\psline[linewidth=0.2,arrowsize=1.2]{->}(3,11.5)(2.8,8.5)
\pscurve[linewidth=0.2,border=0.6](2.8,8.5)(2,4)(-1.5,1.5)
\psline[linewidth=0.2,arrowsize=1.2]{->}(0,12.5)(-0.5,10)
\pscurve[linewidth=0.2,border=0.6](-0.5,10)(-1.2,6)(-1.5,1.5)
\pscurve[linewidth=0.2](2.8,8.5)(2,4)(-1.5,1.5)
\psline[linewidth=0.2,arrowsize=1.2]{->}(-1.5,1.5)(-3,-1.5)
\rput(-3,12.5){$a$}
\rput(0,13.5){$b$}
\rput(3,12.5){$c$}
\rput(-4,-2.5){$b \op c$}
\rput(5.3,-2.7){$(a \lhd b) \lhd c$}
\rput{-55}(0.5,5.5){$a \lhd b$}
\psline[linewidth=0.05]{|->}(-1,1.1)(6,1.1)
\rput(11,0.9){\psframebox{$-\wl(b, c)$}}
\psline[linewidth=0.05]{|->}(0,8.2)(6,8.2)
\rput(11,8.2){\psframebox{$\w(a,b)$}}
\psline[linewidth=0.05]{|->}(3,4)(6,4)
\rput(11,4.1){\psframebox{$\w(a \lhd b,c)$}}
\end{pspicture}

\medskip
\begin{pspicture}(-14.5,-4.5)(17,11.5)
\pscurve[linewidth=0.2,arrowsize=1.2]{->}(-3,9.5)(2,5)(0,2)
\pscurve[linewidth=0.2,border=0.6,arrowsize=1.2]{->}(3,9.5)(-2,5)(0,2)
\psline[linewidth=0.2,arrowsize=1.2]{->}(1,3)(0,2)(0,-2.5)
\rput(-3,10.5){$a$}
\rput(0,-3.5){$b \op (a \lhd b)$}
\rput(3,10.5){$b$}
\rput(4.5,4.5){$a \lhd b$}
\psline[linewidth=0.05]{|->}(-1,2)(-5,2)
\rput(-11,2){\psframebox{$-\wl(b, a \lhd b)$}}
\psline[linewidth=0.05]{|->}(-1,7.5)(-6,7.5)
\rput(-11,7.5){\psframebox{$\w(a,b)$}}
\rput(13,5){$\overset{\text{R}\mathrm{V}^z}{\longleftrightarrow}$}
\end{pspicture}
\begin{pspicture}(-7.5,-4.5)(25,11.5)
\psline[linewidth=0.2,arrowsize=1.2]{->}(0,2)(0,-2.5)
\pscurve[linewidth=0.2,arrowsize=1.2]{->}(-3,9.5)(-2,5)(0,2)
\pscurve[linewidth=0.2,arrowsize=1.2]{->}(3,9.5)(2,5)(0,2)
\rput(-3,10.5){$a$}
\rput(0,-3.5){$a \op b$}
\rput(3,10.5){$b$}
\psline[linewidth=0.05]{|->}(1,2)(5,2)
\rput(11,2){\psframebox{$-\wl(a,b)$}}
\end{pspicture}

\medskip
\begin{pspicture}(-12.5,-3.5)(18,14.5)
\psline[linewidth=0.2,arrowsize=1.2]{->}(-3,11.5)(-2.5,9)
\pscurve[linewidth=0.2](-2.5,9)(-2.3,8)(-1.5,6)(-0.5,5)
\psline[linewidth=0.2,arrowsize=1.2]{->}(0,12.5)(0.5,10)
\pscurve[linewidth=0.2](0.5,10)(0.6,9.5)(0.4,7)(-0.5,5)
\psline[linewidth=0.2,arrowsize=1.2]{->}(-0.5,5)(3,-1.5)
\pscurve[linewidth=0.2,border=0.6,arrowsize=1.5]{->}(3,11.5)(1.5,1.5)(-3,-1.5)
\rput(-3,12.5){$a$}
\rput(0,13.5){$b$}
\rput(3,12.5){$c$}
\rput(-4,-2.5){$c$}
\rput(-1.5,3){$a \op b$}
\rput(5.3,-2.7){$(a \op b) \lhd c$}
\psline[linewidth=0.05]{|->}(-2.5,5)(-6,5)
\rput(-11,5){\psframebox{$-\wl(a,b)$}}
\psline[linewidth=0.05]{|->}(0,1.5)(-6,1.5)
\rput(-11,1.5){\psframebox{$\w(a \op b, c)$}}
\rput(14,5){$\overset{\text{R}\mathrm{VI}^z}{\longleftrightarrow}$}
\end{pspicture}
\begin{pspicture}(-7.5,-3.5)(25,14.5)
\psline[linewidth=0.2,arrowsize=1.2]{->}(-3,11.5)(-2.8,8.5)
\pscurve[linewidth=0.2](-2.8,8.5)(-2,4)(1.5,1.5)
\psline[linewidth=0.2,arrowsize=1.2]{->}(0,12.5)(0.5,10)
\pscurve[linewidth=0.2](0.5,10)(1.2,6)(1.5,1.5)
\psline[linewidth=0.2,arrowsize=1.2]{->}(1.5,1.5)(3,-1.5)
\pscurve[linewidth=0.2,border=0.6,arrowsize=1.2]{->}(3,11.5)(-1,6)(-2,4)(-3,-1.5)
\rput(-3,12.5){$a$}
\rput(0,13.5){$b$}
\rput(3,12.5){$c$}
\rput(-4,-2.5){$c$}
\rput(5.3,-2.7){$(a \lhd c) \op (b \lhd c)$}
\rput(3.5,6){$b \lhd c$}
\rput{-30}(-0.3,1.5){$a \lhd c$}
\psline[linewidth=0.05]{|->}(2.5,1.1)(8,1.1)
\rput(16,0.9){\psframebox{$-\wl(a \lhd c, b \lhd c)$}}
\psline[linewidth=0.05]{|->}(1.5,8.2)(11,8.2)
\rput(16,8.2){\psframebox{$\w(b,c)$}}
\psline[linewidth=0.05]{|->}(-1,4)(11,4)
\rput(16,4.1){\psframebox{$\w(a,c)$}}
\end{pspicture}
\psset{unit=1mm}

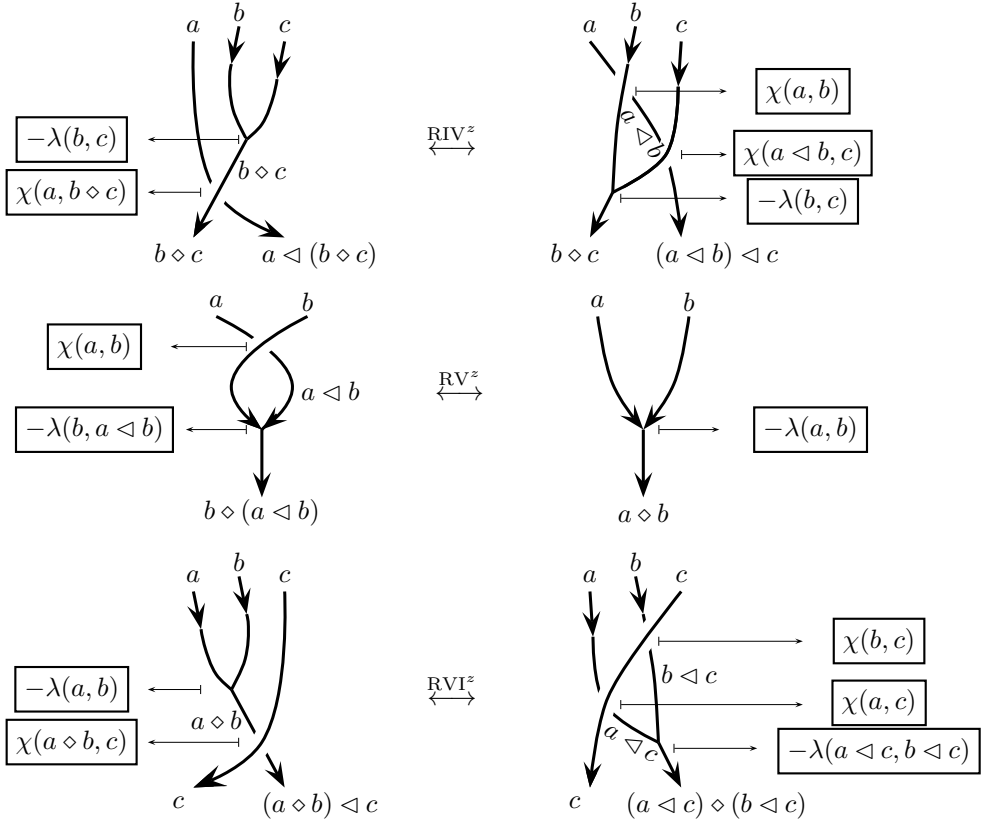
\captionof{figure}{Obtaining axioms for qualgebra $2$-cocycles}\label{pic:AxiomsQu2Cocycles}
\end{center}
\end{proof}

\begin{definition}\label{D:Qualgebra2Coc}
For a qualgebra~$\Q$, a pair of maps $(\w,\wl)$ satisfying the five axioms above is called a ($\ZZ$-valued) \emph{qualgebra $2$-cocycle} of~$\Q$; the term will be commented on below. The set of all qualgebra $2$-cocycles of~$\Q$ is denoted by $Z^2(\Q)$.
\end{definition}

Lemma~\ref{L:WeightedInvar} now allows us to construct weight qualgebra coloring invariants for graphs.
 
\begin{corollary}\label{L:QualgebraWeightInvar}
Take a qualgebra $(\Q,\lhd,\op)$ and a qualgebra $2$-cocycle $(\w,\wl)$. Consider $\Q$-coloring rules from Figure~\ref{pic:Colorings}\rcircled{A}\&\rcircled{B} and the weight function~$(\w,\wl)$ from Figures~\ref{pic:Quandles} and~\ref{pic:BoltzmannGraph}. Then the multi-set $\{\BW_{(\w,\wl)}(\D,\C) | \C \in \Col_{\Q}(\D)\}$ does not depend on the choice of a diagram~$\D$ representing a well-oriented $3$-valent knotted graph~$\Gamma$.
\end{corollary} 

\begin{proof}
Proposition~\ref{P:QualgebraColor} guarantees that the above coloring rules are topological,
and Proposition~\ref{P:BoltzmannGraph} tells that the above weight function is Boltzmann.
Lemma~\ref{L:WeightedInvar} then asserts that the multi-set in question is well-defined on R-equivalence classes of diagrams, which correspond to isotopy classes of graphs.
\end{proof}

One thus gets a systematic way of producing invariants of well-oriented (or unoriented, cf. Proposition~\ref{P:UnVsWellOrient}) graphs, which sharpen the counting invariants from Corollary~\ref{L:QualgebraCountInvar}.

\subsection*{More on qualgebra $2$-cocycles: properties and examples}

Start with an easy observation concerning the structure of $Z^2(\Q)$:

\begin{lemma}\label{L:QualgebraCocGroup}
The space $Z^2(\Q)$ of qualgebra $2$-cocycles of a qualgebra~$\Q$ is an Abelian group under the point-wise addition of the two components; in other words, the sum $(\w,\wl)$ of $(\w',\wl')$ and $(\w'',\wl'')$ is defined by
\begin{align*}
\w(a,b)& = \w'(a,b) + \w''(a,b), & \wl(a,b)& = \wl'(a,b) + \wl''(a,b).
\end{align*} 
Moreover, for a fixed $\Q$-colored graph diagram~$(\D,\C)$, the following map is linear:
\begin{align*}
Z^2(\Q) & \longrightarrow \ZZ, \\
(\w,\wl) & \longmapsto \BW_{(\w,\wl)}(\D,\C).
\end{align*} 
\end{lemma}

\begin{proof}
An easy standard verification using, for the first assertion, the linearity of all qualgebra $2$-cocycle axioms, and, for the second assertion, the linearity  of our qualgebra coloring rules.
\end{proof}

We continue the generalities about qualgebra $2$-cocycles with a remark on their definition. 
Recall that in the definition of a qualgebra, the self-distributivity axiom turned out to be redundant; here, some axioms can be omitted as well:

\begin{lemma}\label{L:QualgebraCocAxioms}
Take a qualgebra~$\Q$ and two maps $\w,\wl:\Q\times\Q\rightarrow\ZZ$. Relation~\eqref{E:BWR3} for these maps follows from~\eqref{E:BWRVI} and~\eqref{E:BWRV}, and relation~\eqref{E:BWR1} is a consequence of~\eqref{E:BWRV}.
\end{lemma}

\begin{proof}
Putting $b=a$ in~\eqref{E:BWRV} and using the idempotence of~$a$, one gets~\eqref{E:BWR1}.

To deduce~\eqref{E:BWR3} from~\eqref{E:BWRVI} and~\eqref{E:BWRV}, one can either use a direct computation, or a diagrammatic argument. We opt for the latter. Consider a sequence of moves R$\mathrm{VI}$ and R$\mathrm{V}$ from Figure~\ref{pic:CocyclesOmittingRIII}. Endow the first and the last diagrams from the figure with the unique colorings extending the partial ones indicated on the Figure, and the intermediate diagrams with the induced colorings (cf. Proposition~\ref{P:QualgebraColor}). Relations~\eqref{E:BWRVI} and~\eqref{E:BWRV} imply, according to the proof of Proposition~\ref{P:BoltzmannGraph}, that the $(\w,\wl)$-weights of all five diagrams coincide. But the $(\w,\wl)$-weights of the first and the last diagrams, decreased by $\wl(a,b)$, are precisely the $(\w,\wl)$-weights of the two sides of a move R$\mathrm{III}$ with colors $a,b,c$ on the top. Recalling that move R$\mathrm{III}$ preserves the $(\w,\wl)$-weights if and only if~\eqref{E:BWR3} holds (cf. Example~\ref{EX:QuandleCol2}), we finish the proof.

\begin{center}
\psset{xunit=1.2mm}
\begin{pspicture}(-4,-2.5)(20,16)
\pscurve[linewidth=0.2,arrowsize=1.2]{->}(-1.5,14)(-3,11.5)(1.5,5)(3,-1.5)
\pscurve[linewidth=0.2,border=0.5,arrowsize=1.2]{->}(-1.5,14)(0,11.5)(-1.5,5)(0,-1.5)
\pscurve[linewidth=0.2,border=0.5,arrowsize=1.2]{->}(3,17)(3,11.5)(1.5,5)(-3,-1.5)
\psline[linewidth=0.2,arrowsize=1.2]{->}(-1.5,17)(-1.5,14)
\psline[linewidth=0.2](-1.5,14)(-2.5,12.8)
\rput(-3.5,13){$a$}
\rput(0.5,13){$b$}
\rput(4,13){$c$}
\rput(13,8){$\overset{\text{R}\mathrm{V}^u}{\longleftrightarrow}$}
\end{pspicture}
\begin{pspicture}(-4,-2.5)(20,16)
\psline[linewidth=0.2,arrowsize=1.2]{->}(-1.5,17)(-1.5,14)
\pscurve[linewidth=0.2,arrowsize=1.2]{->}(-1.5,14)(0,11.5)(1.5,5)(3,-1.5)
\pscurve[linewidth=0.2,arrowsize=1.2]{->}(-1.5,14)(-3,11.5)(-1.5,5)(0,-1.5)
\pscurve[linewidth=0.2,border=0.5,arrowsize=1.2]{->}(3,17)(3,11.5)(1.5,5)(-3,-1.5)
\rput(13,8){$\overset{\text{R}\mathrm{VI}^u}{\longleftrightarrow}$}
\end{pspicture}
\begin{pspicture}(-4,-2.5)(20,16)
\psline[linewidth=0.2,arrowsize=1.2]{->}(-1.5,17)(1.5,6)
\pscurve[linewidth=0.2,arrowsize=1.2]{->}(1.5,6)(2.5,4)(3,-1.5)
\pscurve[linewidth=0.2,arrowsize=1.2]{->}(1.5,6)(0.5,4)(0,-1.5)
\pscurve[linewidth=0.2,border=0.5,arrowsize=1.2]{->}(3,17)(-1,12)(-3,-1.5)
\rput(13,8){$\overset{\text{R}\mathrm{V}^u}{\longleftrightarrow}$}
\end{pspicture}
\begin{pspicture}(-4,-2.5)(20,16)
\psline[linewidth=0.2,arrowsize=1.2]{->}(-1.5,17)(1.5,6)
\pscurve[linewidth=0.2,arrowsize=1.2]{->}(1.5,6)(0.5,4)(3,-1.5)
\pscurve[linewidth=0.2,border=0.5,arrowsize=1.2]{->}(1.5,6)(2.5,4)(0,-1.5)
\pscurve[linewidth=0.2,border=0.5,arrowsize=1.2]{->}(3,17)(-1,12)(-3,-1.5)
\psline[linewidth=0.2](1.5,6)(1.2,5.5)
\rput(13,8){$\overset{\text{R}\mathrm{VI}^u}{\longleftrightarrow}$}
\end{pspicture}
\begin{pspicture}(-4,-2.5)(6,16)
\psline[linewidth=0.2,arrowsize=1.2]{->}(-1.5,17)(-1.5,14)
\pscurve[linewidth=0.2,border=0.5,arrowsize=1.2]{->}(-1.5,14)(-3,11.5)(-1.5,5)(3,-1.5)
\pscurve[linewidth=0.2,border=0.5,arrowsize=1.2]{->}(-1.5,14)(0,11.5)(1.5,5)(0,-1.5)
\pscurve[linewidth=0.2,border=0.5,arrowsize=1.2]{->}(3,17)(3,11.5)(-1.5,5)(-3,-1.5)
\psline[linewidth=0.2](-1.5,14)(-2.5,12.7)
\rput(-3.5,13){$a$}
\rput(0.5,13){$b$}
\rput(4,13){$c$}
\end{pspicture}
\psset{xunit=1mm}

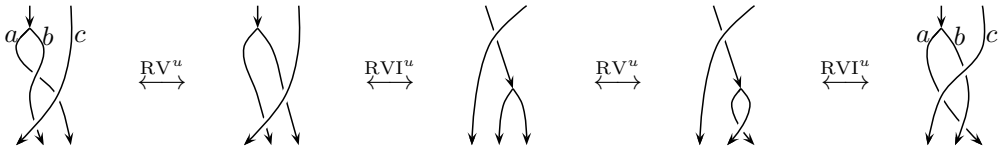
\captionof{figure}{Move R$\mathrm{III}$ as a sequence of moves R$\mathrm{VI}$ and R$\mathrm{V}$}\label{pic:CocyclesOmittingRIII}
\end{center}
\end{proof}

It is thus sufficient to keep only Axioms \eqref{E:BWRIV}-\eqref{E:BWRV} in the definition of qualgebra $2$-cocycles, simplifying their investigation. 

\begin{example}\label{EX:2cocCountVertices}
Let us explore qualgebra $2$-cocycles with a zero part $\w$. In this situation, Axioms \eqref{E:BWRIV}-\eqref{E:BWRV} become
\begin{align}
& \wl(a \lhd c, b \lhd c) = \wl(a,b), \label{E:BWRVI_}\\ 
& \wl(a,b) = \wl(b,a \lhd b). \label{E:BWRV_}
\end{align}
Relation~\eqref{E:BWRVI_} implies that $\wl(b,a \lhd b) = \wl(b \lhd b,a \lhd b) = \wl(b,a)$, thus the maps $(0,\wl)$ form a $2$-cocycle if and only if $\wl$ is a symmetric invariant (in the sense of~\eqref{E:BWRVI_}) form on~$\Q$. The simplest example of such a form is the constant map $\wl_1(a,b)=1$; in this case, $\BW_{(0,\wl_1)}(\D,\C)$ does not depend on the coloring~$\C$ and counts the difference between the numbers of unzip and zip vertices. Another example is the Kronecker delta  $\delta(a,b) = \begin{cases} 1 \text{ if } a = b \\ 0 \text{ otherwise} \end{cases}$, which is also a symmetric invariant form; in this case, $\BW_{(0,\delta)}(\D,\C)$ counts the difference between the numbers of $\C$-isosceles unzip and zip vertices (cf. Definition~\ref{D:IsoscelesColor}).
\end{example}

\subsection*{Qualgebra $2$-cocycles for trivial qualgebras}

We next explicitly describe the structure of $Z^2(\Q)$ for trivial qualgebras~$\Q$ (Definition~\ref{D:TrivialQuandle}):

\begin{proposition}\label{P:TrivialQuandleCoc}
Take a trivial qualgebra $(\Q,\lhd_0,\op)$. Endow~$\Q$ with an arbitrary linear order. Let $\ASF(\Q)$ be the Abelian group of all anti-symmetric bilinear forms~$\asf$ on~$\Q$ (i.e., $\asf(a,b)+\asf(b,a)=0$ and $\asf(a,b\op c) = \asf(a,b) + \asf(a,c)$), and let $\SF(\Q)$ be the Abelian group of all symmetric forms~$\wl$ on~$\Q$ (i.e.,
$\wl(a,b) = \wl(b,a)$). Then the Abelian group $Z^2(\Q)$ of qualgebra $2$-cocycles of~$\Q$ is a direct sum of $\LL = \{\L_{\wl} = (0,\wl) \,|\, \wl \in \SF(\Q) \}$ and of $\XX =\{\X_{\asf}= (\asf,\wl_{\asf}) \,|\, \asf \in \ASF(\Q)\}$, where 
\begin{align*}
\wl_{\asf}(a,b) &= \begin{cases} 0 \text{ if } a \le b, \\ \asf(b,a) \text{ otherwise}.\end{cases}
\end{align*}
\end{proposition}

\begin{proof}
According to Lemma~\ref{L:QualgebraCocAxioms}, we are looking for  maps $\w,\wl:\Q\times\Q\rightarrow\ZZ$ satisfying Axioms \eqref{E:BWRIV}-\eqref{E:BWRV}. Using the triviality of the quandle operation~$\lhd_0$, and renaming the variables in~\eqref{E:BWRVI}, rewrite the Axioms as
\begin{align}
&\w(a,b \op c) = \w(a,b) + \w(a,c),\label{E:Tr1}\\
&\w(b\op c, a) = \w(b,a) +\w(c,a),\label{E:Tr2}\\ 
&\w(a,b) = \wl(b,a) - \wl(a,b).\label{E:Tr3}
\end{align}
The last one implies that~$\w$ is anti-symmetric, which makes~\eqref{E:Tr2} a consequence of~\eqref{E:Tr1}, and also shows that $\w \in \ASF(\Q)$. It suffices thus to consider Axioms~\eqref{E:Tr1} and~\eqref{E:Tr3} only.
Maps $\X_{\asf}$ and $\L_{\wl}$ are easily checked to satisfy these relations. Moreover, $\LL$ is a subgroup of $Z^2(\Q)$ by construction, and so is $\XX$, since $\X_{\asf} + \X_{\asf'} = \X_{\asf+\asf'}$. The intersection of~$\XX$  and $\LL$ is trivial: indeed, $\X_{\asf} = \L_{\wl}$ implies $\asf = 0$, hence $\X_{\asf} = 0$. To see that the two generate the whole $Z^2(\Q)$, note that, as shown above, one has $\w \in \ASF(\Q)$ for any $(\w,\wl) \in Z^2(\Q)$; then $(\w,\wl) - \X_{\w}$ is of the form $(0,\wl')$ and still lies in $Z^2(\Q)$, so, due to~\eqref{E:Tr3}, it satisfies $\wl'(a,b) = \wl'(b,a)$, hence $(0,\wl') = \L_{\wl'}$.
\end{proof}

One thus gets an Abelian group isomorphism $Z^2(\Q) \cong \SF(\Q) \oplus \ASF(\Q)$ for any trivial qualgebra~$\Q$.

\begin{example}\label{EX:2cocCountVerticesTrivial}
Returning to Example~\ref{EX:2cocCountVertices}, one sees that the part~$\LL$ of $Z^2(\Q)$ always contains the cocycles $\L_{\wl_1}$ and $\L_{\delta}$, where $\wl_1(a,b)=1$ for all the arguments, and~$\delta$ is the Kronecker delta.
\end{example}

Note that if~$\Q$ is finite, then the part~$\LL$ of $Z^2(\Q)$ has a basis
$\{\L_{x,y} = (0, \wl_{x,y}) \,|\, x \le y \}$, where $\wl_{x,y}$ takes value $1$ on (perhaps coinciding) pairs $(x,y)$ and $(y,x)$, and value $0$ elsewhere. Moreover, for finite~$\Q$ the part~$\XX$ of $Z^2(\Q)$ becomes trivial, since $\ASF(\Q)$ contains the zero map only: indeed, for $\w \in \ASF(\Q)$ the bilinearity implies that $\w(a,b \op b) = 2\w(a,b)$, thus if~$\w$ takes a non-zero value $\w(a,b)$, then it also takes arbitrary large (or small) values $2^k\w(a,b)$, $k \in \NN$, which is impossible since the set of values of~$\w$ is finite for finite~$\Q$. However, the part~$\XX$ can be non-trivial even for finite~$\Q$ if the coefficients live, for instance, in a finite cyclic group instead of~$\ZZ$.

\subsection*{Qualgebra $2$-cocycles for size $4$ qualgebras}

We now study the structure of $Z^2(\P)$ for non-trivial $4$-element qualgebras~$\P$, classified above:

\begin{proposition}\label{P:Z2Qualgebras4}
Let $(\P,\lhd,\op)$ be any of the nine $4$-element qualgebras from Proposition~\ref{P:Qualgebras4}. Then the group of its $2$-cocycles $Z^2(\P)$ is free Abelian of rank~$8$.
\end{proposition}

\begin{proof}
Lemma~\ref{L:QualgebraCocAxioms} tells us to look for  maps $\w,\wl:\P\times\P\rightarrow\ZZ$ satisfying Axioms \eqref{E:BWRIV}-\eqref{E:BWRV}. 

Start with Axiom~\eqref{E:BWRIV}. For $c=r$ and $b \neq r$, one has $b \op c =r$ and $a \lhd b = a$, so~\eqref{E:BWRIV} is equivalent to $\w(a,b) = 0$. One gets the first relation describing $2$-cocycles:
\begin{align}\label{E:2for4_1}
\forall \: x, \; \forall \: y \neq r, \qquad \w(x,y) &= 0.
\end{align} 
Case $b=r$, $c \neq r$, leads to the same relation. Further, for $b,c \neq r$, their product $b \op c$ is also different from~$r$, so~\eqref{E:2for4_1} implies~\eqref{E:BWRIV}. In the remaining case $b = c =r$, one gets $\w(a,r \op r) = \w(a,r) + \w(a\lhd r,r)$. The right side simplifies as $\w(a,r) + \w((a)\rrho,r)$, the left one is $\w(a,r \op r) = \w(a,s) = 0$ due to~\eqref{E:2for4_1}. One obtains 
\begin{align}
\w(p,r) + \w(q,r) &= 0, \label{E:2for4_2}\\
2\w(r,r) = 2\w(s,r) &= 0. \notag
\end{align}
We choose not to remove the coefficients $2$ in the last relation, so that our argument remains valid for $2$-cocycles with coefficients in any Abelian group; in any case, relation $\w(r,r) = \w(s,r)=0$ will be obtained below without assumptions on the coefficient group.

We now turn to Axiom~\eqref{E:BWRV}. If $b \neq r$, then $a \lhd b = a$, and, using~\eqref{E:2for4_1}, our axiom becomes $\wl(a,b) = \wl(b,a)$. This relation also holds true for $b=r, a \neq r$ by a symmetry argument, and trivially for $a=b=r$. Summarizing, one gets 
\begin{align}\label{E:2for4_3}
\forall \: x, y, \qquad \wl(x,y) &= \wl(y,x).
\end{align}
For $b=r$, \eqref{E:BWRV} becomes $\w(a,r) = \wl(r,(a)\rrho) - \wl(a,r)$, or, separating different values of~$a$ and using the symmetry~\eqref{E:2for4_3} of~$\wl$,
\begin{align}
\w(r,r) = \w(s,r) &= 0,\label{E:2for4_4}\\
\wl(p,r) - \wl(q,r) &= \w(q,r),\label{E:2for4_5}
\end{align}
and $\wl(q,r) - \wl(p,r) = \w(p,r)$, which is a consequence of~\eqref{E:2for4_5} and~\eqref{E:2for4_2} and is thus discarded.

It remains to analyze Axiom~\eqref{E:BWRVI}. For $c \neq r$ or for $c=r$ with $a,b \in \{r,s\}$, one has $a \lhd c = a$ and $b \lhd c = b$, so everything becomes zero due to~\eqref{E:2for4_1}. Consider now the case $c=r$. If $\{a,b\} = \{p,q\}$, then $a \op b = s$ (hence $\w(a \op b,c) = \w(s,r) = 0$ due to~\eqref{E:2for4_4}), $\w(a,r) + \w(b,r) = 0$ because of~\eqref{E:2for4_2}, and $\wl(a \lhd c, b \lhd c) = \wl((a)\rrho,(b)\rrho) = \wl(b,a) = \wl(a,b)$; all of these together imply our axiom. If $a=b=q$, then one gets 
\begin{align}
\wl(p,p) - \wl(q,q) &= 2\w(q,r) - \w(q \op q,r).\label{E:2for4_6}
\end{align}
Case $a=b=p$ leads to the same relation due to~\eqref{E:2for4_2}. For $a=r$, $b \in \{p,q\}$, one has $a \op b = r$, and our axiom becomes $\wl(r,(b)\rrho) = \w(b,r) + \wl(r,b)$, which is equivalent to~\eqref{E:2for4_5} (due to~\eqref{E:2for4_2} and~\eqref{E:2for4_3}). Case $b=r$, $a \in \{p,q\}$ is analogous. If $a=s$ and $b=q$, then our axiom becomes $\w(s \op q,r) +  \wl(s,p) = \w(q,r) + \wl(s,q)$, or else
\begin{align}
\wl(p,s) - \wl(q,s) &= \w(q,r) - \w(q \op s,r).\label{E:2for4_7}
\end{align}
Cases $a=s$, $b = p$ or $b=s$, $a \in \{p,q\}$ lead to the same relation.

Putting everything together, one concludes that $(\w,\wl)$ is a $2$-cocycle for~$\P$ if and only if the maps $\w,\wl:\P\times\P\rightarrow\ZZ$ satisfy Relations \eqref{E:2for4_1}-\eqref{E:2for4_7}. Note that $\w(q \op q,r)$ equals $\w(q,r)$, $-\w(q,r)$ or $0$, according to $q \op q$ being chosen as $q$, $p$ or $s$, and similarly for $\w(q \op s,r)$. Thus, one sees that the $8$ values $\w(q,r)$, $\wl(q,r)$, $\wl(q,s)$, $\wl(q,q)$, $\wl(q,p)$, $\wl(r,r)$, $\wl(s,r)$ and $\wl(s,s)$ can be chosen arbitrarily, and the other values of~$\w$ and~$\wl$ are deduced from these in a unique way. This shows that $Z^2(\P)$ is a free Abelian group of rank~$8$. Indeed, its $i$th generator can be obtained by letting the $i$th of the above values be~$1$, declaring the other $7$ values zero, and calculating the remaining values of~$\w$ and~$\wl$ using Relations \eqref{E:2for4_1}-\eqref{E:2for4_7}.
\end{proof}

\begin{notation}\label{N:BasisQualgebraCoc4}
We denote by $(\b^\w_{q,r}, \b_{q,r}, \b_{q,s}, \b_{q,q}, \b_{q,p}, \b_{r,r}, \b_{s,r}, \b_{s,s})$ the basis of $Z^2(\P)$ obtained in the proof.
\end{notation}


\subsection*{Qualgebra $2$-coboundaries}

Recall the definition $\w_\f(a, b) = \f(a\lhd b) - \f(a)$ of a $2$-coboundary for quandles, with an arbitrary map $\f:\Q \rightarrow \ZZ$ (Example~\ref{EX:QuandleCol2}). It can be interpreted as the difference between the total weight $\f(b) + \f(a\lhd b)$ at the bottom of the diagram describing the quandle coloring rule around a crossing point, and the total weight $\f(a) + \f(b)$ at the top of this diagram (see Figure~\ref{pic:Colorings}\rcircled{A}). Trying to treat the coloring rule around a $3$-valent vertex (Figure~\ref{pic:Colorings}\rcircled{B}) in a similar way, one gets a good candidate for the notion of qualgebra $2$-coboundary:

\begin{definition}\label{D:Qualgebra2Cob}
For a qualgebra~$\Q$ and a  map $\f:\Q \rightarrow \ZZ$, the pair of maps $(\w_\f,\wl_\f)$ defined by
\begin{align*}
\w_\f(a, b)& = \f(a\lhd b) - \f(a),\\
\wl_\f(a, b)& = \f(a) + \f(b) -\f(a\op b) 
\end{align*}
is called a ($\ZZ$-valued) \emph{qualgebra $2$-coboundary} of~$\Q$. The set of all qualgebra $2$-coboundaries of~$\Q$ is denoted by $B^2(\Q)$.
\end{definition}

\begin{proposition}\label{P:Qualgebra2Cob} 
Given a qualgebra~$(\Q,\lhd,\op)$, the set of its qualgebra $2$-coboundaries $B^2(\Q)$ is an Abelian subgroup of $Z^2(\Q)$. Moreover, for any $\Q$-colored graph diagram $(\D,\C)$ and any $2$-coboundary $(\w,\wl)$, the weight $\BW_{(\w,\wl)}(\D,\C)$ is zero.
\end{proposition}

Before giving a proof, we write explicitly the weights of crossing points and vertices constructed out of the maps~$\w_\f$ and~$\wl_\f$ according to the rules from Figures~\ref{pic:Quandles} and~\ref{pic:BoltzmannGraph}; see Figure~\ref{pic:WeigtsCob}.

\begin{center}
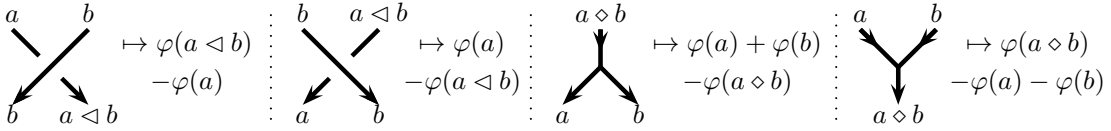

\begin{pspicture}(37,13)
\psline[linewidth=0.6,arrowsize=2]{->}(0,10)(10,0)
\psline[linewidth=0.6,border=1.8,arrowsize=2]{->}(10,10)(0,0)
\rput[b](0,11){$a$}
\rput[b](10,11){$b$}
\rput[t](0,0){$b$}
\rput[t](10,0){$a \lhd b$}
\rput(23,8){$\mapsto \f(a \lhd b)$}
\rput(23,3){$ -\f(a)$}
\psline[linestyle=dotted](34,-2)(34,12)
\end{pspicture}
\begin{pspicture}(33,13)
\psline[linewidth=0.6,arrowsize=2]{->}(10,10)(0,0)
\psline[linewidth=0.6,border=1.8,arrowsize=2]{->}(0,10)(10,0)
\rput[b](0,11){$b$}
\rput[b](10,11){$a \lhd b$}
\rput[t](0,0){$\bw a \bw$}
\rput[t](10,0){$b$}
\rput(21,8){$\mapsto \f(a)$}
\rput(21,3){$ -\f(a \lhd b)$}
\psline[linestyle=dotted](30,-2)(30,12)
\end{pspicture}
\begin{pspicture}(38,13)
\psline[linewidth=0.6,arrowsize=2]{->}(5,10)(5,7)
\psline[linewidth=0.6,arrowsize=2]{->}(5,10)(5,5)(0,0)
\psline[linewidth=0.6,arrowsize=2]{->}(5,5)(10,0)
\rput[b](5,11){$a \op b$}
\rput[t](0,0){$\bw a \bw$}
\rput[t](10,0){$b$}
\rput(23,8){$\mapsto \f(a) + \f(b)$}
\rput(23,3){$- \f(a \op b)$}
\psline[linestyle=dotted](36,-2)(36,12)
\end{pspicture}
\begin{pspicture}(30,13)
\psline[linewidth=0.6,arrowsize=2]{->}(0,10)(2.5,7.5)
\psline[linewidth=0.6,arrowsize=2]{->}(0,10)(5,5)(5,0)
\psline[linewidth=0.6,arrowsize=2]{->}(10,10)(7.5,7.5)
\psline[linewidth=0.6](10,10)(5,5)
\rput[b](0,11){$a$}
\rput[b](10,11){$b$}
\rput[t](5,0){$a \op b$}
\rput(22,8){$\mapsto \f(a \op b)$}
\rput(22,3){$ - \f(a) - \f(b)$}
\end{pspicture}

\medskip
\captionof{figure}{Weight function for maps~$\w_\f$ and~$\wl_\f$}\label{pic:WeigtsCob}
\end{center}

\begin{proof}
Let us first show that a qualgebra $2$-coboundary $(\w_\f,\wl_\f)$ of~$\Q$ is also a qualgebra $2$-cocycle of~$\Q$. One can either check the necessary Axioms \eqref{E:BWRIV}-\eqref{E:BWRV} directly, or develop the ``total weight increment'' argument which lead to the definition of qualgebra $2$-coboundaries. Indeed, the $(\w_\f,\wl_\f)$-weight (Figure~\ref{pic:WeigtsCob}) of the $\Q$-colored diagrams that appear in R-moves with unzip vertices (Figure~\ref{pic:AxiomsQu2Cocycles}) is the difference between the total $\f$-weight at the bottom and at the top of these diagrams. Since the bottom/top colors are the same for both diagrams involved in an R-move, these diagrams have the same $(\w_\f,\wl_\f)$-weights, which means, according to (the proof of) Proposition~\ref{P:BoltzmannGraph}, that $(\w_\f,\wl_\f)$ is a qualgebra $2$-cocycle.

We have thus showed that $B^2(\Q) \subseteq Z^2(\Q)$. To see that it is an Abelian subgroup, observe that $(\w_\f,\wl_\f) + (\w_{\f'},\wl_{\f'}) = (\w_{\f+\f'},\wl_{\f+\f'})$, where maps $\Q \rightarrow \ZZ$ are added point-wise.

Take now a $\Q$-colored graph diagram $(\D,\C)$ and a $2$-coboundary $(\w_\f,\wl_\f)$. As shown above, the latter is also a $2$-cocycle, and hence, according to Proposition~\ref{P:BoltzmannGraph}, defines a Boltzmann weight function. We shall now prove that the total $\w_\f$-weight of the crossing points of $(\D,\C)$ kills the total $\wl_\f$-weight of its $3$-valent vertices, implying that $\BW_{(\w_\f,\wl_\f)}(\D,\C)=0$.

Consider an edge~$e$ of~$\D$, and analyse how the color behaves when one moves along~$e$. The color changes from~$a$ to $a \lhd b$ or $a \wlhd b$ when~$e$ goes under a $b$-colored arc (depending on the orientation of the latter) and stays constant otherwise. Observing that $\f(a \lhd^{\pm 1} b) - \f(a)$ is precisely the $\w_\f$-weight of the crossing point where the color changes, one concludes that the total weight of all the crossing points of~$\D$ is the sum $\sum_{e} \f(\C(t(e))) - \f(\C(s(e)))$ taken over all the edges~$e$ of~$\D$, where $s(e)$ and $t(e)$ are, respectively, the first and the last arcs of~$e$. Since each edge starts and finishes at a $3$-valent vertex, this sum can be reorganized to the sum $\sum_v \sum_{\alpha \in \A(v)} \pm \f(\alpha)$ taken over all the vertices~$v$ of~$\D$, where $\A(v)$ is the set of arcs adjacent to~$v$, and $\f(\alpha)$ is taken with the sign~$-$ if~$\alpha$ is directed from~$v$, and~$+$ otherwise. On the other hand, the total weight of all the $3$-valent vertices is the sum of the same form, but with the opposite sign convention (see Figure~\ref{pic:WeigtsCob}).
\end{proof}

\begin{example}\label{EX:2CobTrivialQualgebras}
Lets us now describe a qualgebra $2$-coboundary $(\w_\f,\wl_\f)$ for a trivial qualgebra~$(\Q,\lhd_0,\op)$ (Definition~\ref{D:TrivialQuandle}). Its $\w$-component is necessarily zero: $\w_\f(a, b) = \f(a\lhd_0 b) - \f(a) = \f(a) - \f(a) =0$. Its $\wl$-component is a symmetric form $\wl_\f(a, b) = \f(a) + \f(b) -\f(a\op b)$ (recall that~$\op$ is commutative for trivial qualgebras). Thus our $2$-coboundaries have the form  $\L_{\wl_\f}$, where~$\f$ runs through all maps from~$\Q$ to~$\ZZ$, and they all lie in the $\LL$-part of $Z^2(\Q)$ (cf. Proposition~\ref{P:TrivialQuandleCoc}).
\end{example}

\subsection*{Towards a qualgebra homology theory}

Proposition~\ref{P:Qualgebra2Cob} legitimates the following

\begin{definition}\label{D:QualgebraH2}
For a qualgebra~$\Q$, the quotient Abelian group $H^2(\Q) = Z^2(\Q) / B^2(\Q)$ is called the second ($\ZZ$-valued) \emph{qualgebra cohomology} of~$\Q$.
\end{definition}

Moreover, Proposition~\ref{P:Qualgebra2Cob} and Lemma~\ref{L:QualgebraCocGroup} imply that the $\Q$-colored graph diagram weight $\BW_{[(\w,\wl)]}(\D,\C)$ is well defined for equivalence classes $[(\w,\wl)] \in H^2(\Q)$. Note that however this need not be true for sub-diagrams.

We now calculate the second qualgebra cohomology for non-trivial $4$-element qualgebras. Remark that the result is the same for all the nine structures. Note also the torsion appearing in the cohomology groups.

\begin{proposition}\label{P:H2Qualgebras4}
Let $(\P,\lhd,\op)$ be any of the nine $4$-element qualgebras from Proposition~\ref{P:Qualgebras4}. Then one has $B^2(\P) \cong \ZZ^4$ and $H^2(\P) \cong \ZZ/2\ZZ \oplus \ZZ^4$.
\end{proposition}

\begin{proof}
Recall the basis $\B=(\b^\w_{q,r}, \b_{q,r}, \b_{q,s}, \b_{q,q}, \b_{q,p}, \b_{r,r}, \b_{s,r}, \b_{s,s})$  of $Z^2(\P)$ (Notation~\ref{N:BasisQualgebraCoc4}). Consider also the subgroup $Z'$ of $Z^2(\P)$ with basis $\B'=(\b^\w_{q,r}, \b_{q,r}, \b_{q,s}, \b_{q,q}, \b_{q,p}, 2\b_{r,r}, \b_{s,r}-\b_{r,r}, \b_{s,s})$. The ``Dirac maps'' $\f_a:\P \rightarrow \ZZ$, $a \in \P$ defined by $\f_a(x) = \delta(a,x) = \begin{cases} 1 \text{ if } x = a \\ 0 \text{ otherwise} \end{cases}$ form a basis of the Abelian group  of maps $\f: \P \rightarrow \ZZ$, hence the pairs of maps $\b_a=(\w_{\f_a},\wl_{\f_a})$ with $a \in \P$ generate $B^2(\P) \subset Z^2(\P)$. We shall now show that in fact all the~$\b_a$ lie in~$Z'$, and that $\B''=(\b_p, \b_q, \b_r, \b_s, \b_{q,s}, \b_{q,q}, \b_{q,p}, \b_{s,s})$ is an alternative basis of $Z'$; this would give a $4$-element basis $(\b_p, \b_q, \b_r, \b_s)$ of $B^2(\P)$ and a $4$-element basis $([\b_{q,s}], [\b_{q,q}], [\b_{q,p}], [\b_{s,s}])$ of $Z'/B^2(\P)$ (here and afterwards the square brackets stand for equivalence classes of pairs of maps). Moreover, by construction $Z^2(\P)/Z' \cong \ZZ/2\ZZ$, and $[\b_{r,r}]$ is its generator. Putting together all the pieces, one gets
$$H^2(\P) = Z^2(\P)/B^2(\P) \cong  Z^2(\P)/Z' \oplus Z'/B^2(\P) \cong \ZZ/2\ZZ \oplus \ZZ^4.$$

In order to show that $\B''$ is indeed a basis, we calculate for the $2$-coboundaries $\b_a$ the $8$ values which completely determine a $2$-cocycle (cf. the proof of Proposition~\ref{P:Z2Qualgebras4}):

\begin{center}
\begin{tabular}{|c|cccc|cccc|}
\hline 
& $\w(q,r)$ & $\wl(q,r)$ & $\wl(r,r)$ & $\wl(s,r)$ & 
$\wl(q,s)$ & $\wl(q,q)$ & $\wl(q,p)$ & $\wl(s,s)$\\
\hline
$\b_p$ & $1$ & $0$ & $0$ & $0$ & $-\alpha_1$ & $-\beta_1$ & $1$ & $0$ \\
$\b_q$ & $-1$ & $1$ & $0$ & $0$ & $1-\alpha_2$ & $2-\beta_2$ & $1$ & $0$ \\
$\b_r$ & $0$ & $0$ & $2$ & $0$ & $0$ & $0$ & $0$ & $0$ \\
$\b_s$ & $0$ & $0$ & $-1$ & $1$ & $1-\alpha_3$ & $-\beta_3$ & $-1$ & $1$ \\
\hline
\end{tabular}
\captionof{table}{Essential components of the $2$-coboundaries $\b_a$}\label{tab:b_a}
\end{center}

In the table, exactly one $\alpha_i$ and one $\beta_j$ equal $1$, while the other are zero; this depends on the values of $q \op s$ and $q \op q$ in our $\P$.

Adding some linear combinations of the $2$-cocycles $\b_{q,s}$, $\b_{q,q}$, $\b_{q,p}$, and $\b_{s,s}$, one can transform the $\b_a$'s into $2$-cocycles $\ob_a \in Z'$ for which the value table can be obtained from Table~\ref{tab:b_a} by replacing everything in its right part (after the middle vertical bar) by zeroes. Since the $8$ values in the table completely determine a $2$-cocycle, one can express the elements of~$\B'$ in terms of those of~$\B''$: 
\begin{align*}
\b^\w_{q,r} &= \ob_q, & \b_{s,r}-\b_{r,r} &=\ob_s,\\
\b_{q,r} &= \ob_q + \ob_p, &  2\b_{r,r} &= \ob_r.
\end{align*}
Now~$\B''$ is a basis since~$\B'$ is such.
\end{proof}

Observe that the weight invariants corresponding to the $2$-cocycles $\b_{q,s}$, $\b_{q,q}$, $\b_{q,p}$, $\b_{s,s}$, and $\b_{r,r}$, whose classes modulo $B^2(\P)$ generate $H^2(\P)$, have an easy combinatorial description. Namely, $\BW_{\b_{q,s}}(\D,\C)$ counts the difference between the numbers of  unzip and zip vertices whose adjacent co-oriented arcs are colored with either~$s$ and~$p$, or~$s$ and~$q$; $\BW_{\b_{q,p}}(\D,\C)$ counts a similar difference for arcs colored with~$p$ and~$q$; 
finally, $\BW_{\b_{q,q}}(\D,\C)$ (respectively, $\BW_{\b_{s,s}}(\D,\C)$) counts a similar difference for both arcs having the same color~$p$ or~$q$ (respectively, $s$). 
Now, Proposition~\ref{P:Qualgebra2Cob} and Lemma~\ref{L:QualgebraCocGroup} imply that $\b_{r,r}$ gives only trivial invariants, at least when one works over~$\ZZ$ (one has $2\BW_{\b_{r,r}}(\D,\C) = \BW_{2\b_{r,r}}(\D,\C)=0$, the cocycle $2\b_{r,r}$ being a coboundary), and that the four invariants described above contain all the information one can deduce from non-trivial $4$-element qualgebra colorings of graphs using the Boltzmann weight method.

\begin{remark}
One would certainly expect qualgebra $2$-cocycles and $2$-coboundaries described above to fit into a complete \emph{qualgebra cohomology theory}, extending the celebrated quandle cohomology theory. However, the author knows how to construct such a theory for \emph{non-commutative qualgebras} only (that is, one keeps Axioms~\eqref{E:SD}-\eqref{E:Idem} and~\eqref{E:QA1}-\eqref{E:QA2}, but not the ``semi-commutativity''~\eqref{E:QAComm}). Topologically, this structure corresponds to \emph{rigid-vertex} well-oriented $3$-valent graphs, for which move R$\mathrm{V}$ should be removed from the list of Reidemeister moves (cf. also \cite{KauffmanGraphs}). Two-cocycles for this structure are defined by Relations~\eqref{E:BWR3}-\eqref{E:BWR1} and~\eqref{E:BWRIV}-\eqref{E:BWRVI} (Relation~\eqref{E:BWRV} being omitted), and they give precisely the Boltzmann weight functions for rigid-vertex graph diagrams. A cohomology theory for non-commutative qualgebras can be defined using the \textit{braided system} concept from \cite{Lebed2}; we shall present the details in a separate publication.
\end{remark}

\subsection*{Squandle $2$-cocycles}

Weight invariants can also be constructed out of squandle colorings, by a procedure that very closely repeats what we have done for qualgebra colorings. We shall now briefly present relevant definitions and results; all the details and proofs can be easily adapted from the qualgebra case.

\begin{definition}\label{D:Squandle2Coc}
For a squandle~$\Q$, a ($\ZZ$-valued) \emph{squandle $2$-cocycle} of~$\Q$ is a pair of maps $\w:\Q\times\Q\rightarrow\ZZ$, $\wl:\Q \rightarrow\ZZ$ satisfying Axioms \eqref{E:BWR3}-\eqref{E:BWR1} together with two additional ones: 
\begin{align*}
&\w(a,b^2) = \w(a,b) + \w(a\lhd b,b),\\
&\w(a^2, b) + \wl(a \lhd b) = 2\w(a,b) + \wl(a).
\end{align*}
The Abelian group of all squandle $2$-cocycles of~$\Q$ is denoted by $Z^2(\Q)$.
\end{definition}

Note that Axioms~\eqref{E:BWR3}-\eqref{E:BWR1} can no longer be omitted from the definition. 

\begin{proposition}\label{P:BoltzmannGraphSquandle}
Take a squandle~$\Q$ and maps $\w:\Q\times\Q\rightarrow\ZZ$, $\wl:\Q \rightarrow\ZZ$. The weight function constructed out of $(\w,\wl)$ according to Figures~\ref{pic:Quandles} and~\ref{pic:BoltzmannGraphSquandle} is Boltzmann if and only if $(\w,\wl) \in Z^2(\Q)$.
\end{proposition}

\begin{center}
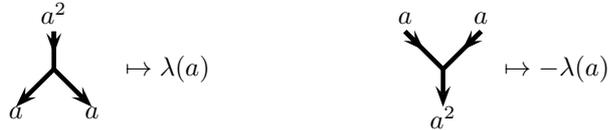

\begin{pspicture}(50,12)
\psline[linewidth=0.6,arrowsize=2]{->}(5,10)(5,7)
\psline[linewidth=0.6,arrowsize=2]{->}(5,10)(5,5)(0,0)
\psline[linewidth=0.6,arrowsize=2]{->}(5,5)(10,0)
\rput[b](5,11){$a^2$}
\rput[t](0,0){$a$}
\rput[t](10,0){$a$}
\rput(20,5){$\mapsto \wl(a)$}
\end{pspicture}
\begin{pspicture}(15,12)
\psline[linewidth=0.6,arrowsize=2]{->}(0,10)(2.5,7.5)
\psline[linewidth=0.6,arrowsize=2]{->}(0,10)(5,5)(5,0)
\psline[linewidth=0.6,arrowsize=2]{->}(10,10)(7.5,7.5)
\psline[linewidth=0.6](10,10)(5,5)
\rput[b](0,11){$a$}
\rput[b](10,11){$a$}
\rput[t](5,0){$a^2$}
\rput(20,5){$\mapsto -\wl(a)$}
\end{pspicture}
\captionof{figure}{Weight function for squandle-colored graph diagrams}\label{pic:BoltzmannGraphSquandle}
\end{center}

\begin{corollary}\label{L:SquandleWeightInvar}
Take a squandle~$\Q$ and a squandle $2$-cocycle $(\w,\wl)$. Consider $\Q$-coloring rules from Figure~\ref{pic:Colorings}\rcircled{A}\&\rcircled{C} and the weight function from Figures~\ref{pic:Quandles} and~\ref{pic:BoltzmannGraphSquandle}, still denoted by $(\w,\wl)$. Then the multi-set $\{\BW_{(\w,\wl)}(\D,\C) \,|\, \C\in\Col_{\Q}(\D)\}$ does not depend on the choice of a diagram~$\D$ representing a well-oriented $3$-valent knotted graph~$\Gamma$.
\end{corollary}

\begin{definition}\label{D:Squandle2Cob}
For a squandle~$\Q$ and a  map $\f:\Q \rightarrow \ZZ$, the pair of maps $(\w_\f,\wl_\f)$ defined by
\begin{align*}
\w_\f(a, b)& = \f(a\lhd b) - \f(a),\\
\wl_\f(a)& = 2\f(a)-\f(a^2) 
\end{align*}
is called a ($\ZZ$-valued) \emph{squandle $2$-coboundary} of~$\Q$. The Abelian group of all squandle $2$-coboundaries of~$\Q$ is denoted by $B^2(\Q)$.
\end{definition}

\begin{proposition}\label{P:Squandle2Cob}
Given a squandle~$\Q$, the set of its squandle $2$-coboundaries $B^2(\Q)$ is a subgroup of $Z^2(\Q)$. Moreover, for any $\Q$-colored graph diagram $(\D,\C)$ and any $2$-coboundary $(\w,\wl)$, the weight $\BW_{(\w,\wl)}(\D,\C)$ is zero.
\end{proposition}

\begin{definition}\label{D:H2}
For a squandle~$\Q$, the quotient Abelian group $H^2(\Q) = Z^2(\Q) / B^2(\Q)$ is called the second ($\ZZ$-valued) \emph{squandle cohomology} of~$\Q$.
\end{definition}

\begin{example}
For trivial squandles, all $2$-cocycles have the form $(\w,\wl)$, where~$\wl$ is arbitrary, and~$\w$ satisfies
\begin{align*}
&\w(a,b^2) = \w(a^2, b) = 2\w(a,b),\\
&\w(a,a)=0.
\end{align*}
In particular, all $\ZZ$-valued $2$-cocycles of finite squandles have a zero $\w$-part.
The $2$-coboundaries have the form $(0,\wl_\f)$ here, where~$\wl_\f(a)=2\f(a)-\f(a^2)$.
\end{example}

\begin{example}
Recall the four $4$-element squandles from Proposition~\ref{P:Squandles4}. Arguments analogous to those used to prove Propositions~\ref{P:Z2Qualgebras4} and~\ref{P:H2Qualgebras4} show that for all these squandles, the Abelian groups $Z^2(\Q)$ and $B^2(\Q)$ are free of rank~$4$. As for cohomologies, one has $H^2(\Q) \cong \ZZ / 2\ZZ$, except for the squandle of the second type with $q^2=s$, in which case one obtains $H^2(\Q) \cong \ZZ / 2\ZZ \oplus \ZZ / 2\ZZ$.
\end{example}

\section{Going further}

This is the first paper in a series of publications devoted to qualgebras and squandles. A lot of work remains to be done on the algebraic as well as on the topological sides.

First, we are currently working on an algebraic study of qualgebras and squandles (\cite{LebedQAAlg}): their general properties, the ``qualgebrization'' of familiar quandles (cf. Example~\ref{EX:NonQualgebras}), conceptual examples, a classification of all structures in small size. Sizes $5$ and $6$ are still doable by hand, and contain a large variety of examples. It would be interesting to calculate the induced invariants for reasonably ``small'' graphs.
Also, as mentioned in Section~\ref{sec:QualgebraHom}, general qualgebra and squandle cohomology theories would be of interest.

There is also a variation of qualgebra/squandle structure called \textit{symmetric qualgebra/squandle}. It includes a special involution~$\rho$ compatible both with the quandle operation~$\lhd$ --- in the sense of Axioms~\eqref{E:SymQu1}-\eqref{E:SymQu3} (thus~$\rho$ is a good involution), and with the qualgebra/squandle operation --- in the sense of certain natural axioms. 
Symmetric quandles were invented by S.Kamada (\cite{KamadaSymm}) in order to extend quandle coloring invariants of oriented knots to unoriented ones; they were later used by Y.Jang and K.Oshiro (\cite{JangOshiro}) for extending quandle coloring invariants of oriented graphs (with coloring rules from Figure~\ref{pic:QuandlesForGraphs}\rcircled{C}) to unoriented ones. Similarly, our symmetric qualgebras/squandles are tailored for coloring unoriented knotted $3$-valent graph diagrams, and therefore lead to invariants of such graphs. Together with the usual group example, one finds numerous examples even in small size. A detailed study of these structures and their topological applications will appear in a subsequent publication.

Lastly, a variation of coloring ideas includes assigning colors to diagram regions, and not only arcs, with a relevant notion of topological coloring rules. Such colorings are called \textit{shadow colorings} in the quandle case, and corresponding counting and weight invariants prove to be extremely powerful for knots. The same can be done for graphs by introducing the notions of \textit{qualgebra/squandle modules} (used for coloring regions), \textit{qualgebra/squandle $2$-cocycles with coefficients} (used for fabricating Boltzmann weight functions), and constructing counting and weight invariants out of these. Note that, regarding a qualgebra/squandle as a module over itself, one naturally gets a definition of \textit{qualgebra/squandle $3$-cocycles} (without coefficients), suggesting one more step towards a qualgebra/squandle cohomology theory. All of this will be presented in details elsewhere.

\appendix
\appendixpage
\section{Proof of Proposition~\ref{P:QualgebraAxiomatize}}\label{A:Proof}

Take two elements $a \neq b$ from~$\SS$. In (the group qualgebra of) $FG_{\SS}$, one has
\begin{equation}\label{E:FreeQA}
(b \lhd a) \op (a \lhd b) = ((a \wlhd b) \lhd a) \op b,
\end{equation} 
since both equal $a^{-1}bab^{-1}ab$. Let us show that this relation fails in $FAQA_{\SS}$.

We first present a detailed description of $FAQA_{\SS}$. The proof is omitted here, but will appear in~\cite{LebedQAAlg}; it is close to what was done in a related context in~\cite{DehornoyFreeALDS}.
A \textit{$\lhd$-term} in $FAQA_{\SS}$ is an element of the form 
\begin{equation*}
t = (\cdots ((a_0 \lhd^{\varepsilon_1} a_1) \lhd^{\varepsilon_2} a_2) \cdots )\lhd^{\varepsilon_r} a_r,
\end{equation*}
where $a_i \in \SS$, $\varepsilon_i \in \{\pm\}$, and, as usual, $\lhd^+$ denotes~$\lhd$, while $\lhd^- = \wlhd$. We compactly write it as
\begin{equation}\label{E:FreeQA2}
t = a_0 \lhd^{\varepsilon_1} a_1 \lhd^{\varepsilon_2} a_2 \cdots \lhd^{\varepsilon_r} a_r.
\end{equation}
A $\lhd$-term is called \textit{reduced} if $a_0 \neq a_1$ and there are no $i > 0$ with $a_i = a_{i+1}$ and $\varepsilon_i = -\varepsilon_{i+1}$. Applying Axioms~\eqref{E:Inv}-\eqref{E:Idem}, seen as rewriting rules here, any $\lhd$-term~$t$ can be presented as a uniquely determined reduced one, denoted by $red(t)$ and called the \textit{reduced form} of~$t$. 

\begin{lemma}\label{L:FreeQA}
\begin{enumerate}
\item Any $x \in FAQA_{\SS}$ can be written in a \emph{product form}, i.e., omitting parentheses thanks to the associativity, as $x = t_1 \op t_2 \op \cdots \op t_n$, where each $t_i$ is a reduced $\lhd$-term. 
\item If an $x \in FAQA_{\SS}$ has two presentations $x = t_1 \op t_2 \op \cdots \op t_n$ and $x = t'_1 \op t'_2 \op \cdots \op t'_{n'}$ as above, then $n = n'$, and the presentations are related by a finite sequence of applications of~\eqref{E:QAComm}. Concretely, a \emph{``positive''} application of~\eqref{E:QAComm} consists in replacing $t_i \op t_{i+1}$ with $t_{i+1} \op red(t_i \lhd t_{i+1})$, and a \emph{``negative''} application replaces it with $red(t_{i+1} \wlhd t_i) \op t_{i}$, where $t_{i+1} \wlhd t_i$ for example is seen as a $\lhd$-term via
\begin{align*}
t_{i+1} \wlhd t_i &= t_{i+1} \wlhd (a_0 \lhd^{\varepsilon_1} a_1 \cdots \lhd^{\varepsilon_r} a_r) = t_{i+1} \lhd^{-\varepsilon_r} a_r \cdots \lhd^{-\varepsilon_1} a_1 \wlhd a_0\lhd^{\varepsilon_1} a_1\cdots \lhd^{\varepsilon_r} a_r.
\end{align*}
\end{enumerate}
\end{lemma}

A reduced $\lhd$-term~$t$ written as in~\eqref{E:FreeQA2} is called a \textit{tail} of a reduced $\lhd$-term~$t'$ if    
\begin{equation}\label{E:FreeQA4}
t' = b_0 \lhd^{\zeta_1} b_1 \cdots \lhd^{\zeta_s} b_s \lhd a_0 \lhd^{\varepsilon_1} a_1 \cdots \lhd^{\varepsilon_r} a_r,
\end{equation}
with the additional technical condition $b_s \neq a_0$. This relation is clearly transitive: a tail of a tail of~$t'$ is still a tail of~$t'$.

This vocabulary allows us to state a lemma crucial for proving the proposition:
\begin{lemma}\label{L:FreeQA2}
Let~$t$ and~$t'$ be reduced $\lhd$-terms such that~$t$ is a tail of~$t'$. Then~$t'$ is a tail of $red(t \lhd t')$.
\end{lemma}

\begin{proof}
Writing~$t$ and~$t'$ as in~\eqref{E:FreeQA2} and~\eqref{E:FreeQA4}, one has
\begin{align*}
t \lhd t' &= (a_0 \lhd^{\varepsilon_1} a_1 \cdots \lhd^{\varepsilon_r} a_r) \lhd (b_0 \lhd^{\zeta_1} b_1 \cdots \lhd^{\zeta_s} b_s \lhd a_0 \lhd^{\varepsilon_1} a_1 \cdots \lhd^{\varepsilon_r} a_r)\\
&= a_0 \lhd^{\varepsilon_1} a_1 \cdots \lhd^{\varepsilon_r} a_r \lhd^{-\varepsilon_r} a_r \cdots \lhd^{-\varepsilon_1} a_1 \wlhd a_0 \lhd^{-\zeta_s} b_s \cdots \lhd^{-\zeta_1} b_1 \lhd b_0 \lhd^{\zeta_1} b_1 \cdots \lhd^{\varepsilon_r} a_r\\
&= a_0 \lhd^{-\zeta_s} b_s \cdots \lhd^{-\zeta_1} b_1 \lhd b_0 \lhd^{\zeta_1} b_1 \cdots \lhd^{\varepsilon_r} a_r.
\end{align*}
The last $\lhd$-term is reduced since~$t'$ is such and since $b_s \neq a_0$ (cf. the definition of a tail). Further, it has~$t'$ as a tail (the technical condition becomes $b_1 \neq b_0$ if $s \ge 1$ --- which follows from the definition of a reduced $\lhd$-term --- and $a_0 \neq b_0$ if $s = 0$ -- which is precisely $b_s \neq a_0$).
\end{proof}

Now, let us return to Relation~\eqref{E:FreeQA}. Both of its sides are written in a product form. Starting with its left-hand side, we shall show that after any number of ``positive'' applications of~\eqref{E:QAComm}, neither of the two $\lhd$-terms becomes $b$; the case of ``negative'' applications is treated similarly, and Lemma~\ref{L:FreeQA} then assures that the two sides of~\eqref{E:FreeQA} represent different elements of $FAQA_{\SS}$.

The first ``positive'' application of~\eqref{E:QAComm} gives
$$(b \lhd a) \op (a \lhd b) = (a \lhd b) \op (b \lhd a \wlhd b \lhd a \lhd b).$$
The $\lhd$-term $a \lhd b$ is a tail of $b \lhd a \wlhd b \lhd a \lhd b$, both of them being reduced.
Now, according to Lemma~\ref{L:FreeQA2} and the transitivity of the tail relation, the $\lhd$-term $a \lhd b$ will be a tail of all the $\lhd$-terms appearing after all further ``positive'' applications of~\eqref{E:QAComm}. Therefore, one never gets the $\lhd$-term~$b$, of which $a \lhd b$ is not a tail.

\bibliographystyle{alpha}
\bibliography{biblio}

\end{document}